\documentclass[a4paper,11pt,leqno]{article}
\usepackage{amssymb, amsmath, amsthm, graphicx}

\newtheorem{thm}{Theorem}[section]
\newtheorem{lem}[thm]{Lemma}
\newtheorem{cor}[thm]{Corollary}
\newtheorem{prop}[thm]{Proposition}
\newtheorem{defn}{Definition}[section]
\newtheorem{rem}{Remark}[section]

\numberwithin{equation}{section}

\renewcommand{\a}{\alpha}
\renewcommand{\b}{\beta}
\newcommand{\e}{\varepsilon}
\newcommand{\de}{\delta}

\newcommand{\la}{\lambda}

\newcommand{\si}{\sigma}
\newcommand{\Si}{\Sigma}

\newcommand{\De}{\Delta}
\newcommand{\Ga}{\Gamma}
\newcommand{\La}{\Lambda}
\newcommand{\Om}{\Omega}

\def\real{{\mathbb R}}

\def\integer{{\mathbb Z}}
\def\torus{{\mathbb T}}

\def\R{{\mathbb{R}}}
\def\N{{\mathbb{N}}}
\def\Z{{\mathbb{Z}}}
\def\T{{\mathbb{T}}}
\def\C{{\chi}}
\def\M{{\mathcal{M}}}
\def\Cy{{\mathcal{C}}}
\def\D{{\mathcal{D}}}

\allowdisplaybreaks

\title{Hydrodynamic limit for
two-species exclusion processes}
\author{Makiko Sasada \\
\small{Graduate School of Mathematical Sciences,} \\
\small{The University of Tokyo, Komaba, Tokyo 153-8914, Japan} \\
\small{e-mail: sasada@ms.u-tokyo.ac.jp} }
\date{}

\begin{document}
\maketitle

\begin{abstract}
\noindent
We consider two-species exclusion processes on the $d$-dimensional discrete torus taking the effects of exchange, creation and annihilation into account. The model is, in general, of nongradient type. We prove that the (charged) particle density converges to the solution of a certain nonlinear diffusion equation under the diffusive rescaling in space and time. We also prove a lower bound on the spectral gap for the generator of the process confined in a finite volume. 
\footnote{ Tel.: +81-3-5465-7001; Fax: +81-3-5465-7011 (M.Sasada).}
\end{abstract}


\section{Introduction}
The aim of this paper is to obtain the hydrodynamic behavior of two-species exclusion processes. Our results can be applied to establish the hydrodynamic limit for the evolution of height differences in interfaces governed by the 1-dimensional SOS dynamics.

The two-species exclusion process describes the evolution of a system of mechanically distinguishable particles, say $+$particles and $-$particles moving on a discrete lattice space under the constraint that at most one particle can occupy each site. The state space of the process is given by $\{-1,0,1\}^{\torus_N^d}$ where $\T^d_N$ stands for the $d$-dimensional discrete torus with side-length N and its elements (called configurations) are denoted by $\eta=(\eta(x), x \in \torus_N^d)$, with $\eta(x)=0$ or $1$ or $-1$ depending on whether $x \in \torus_N^d$ is empty or occupied by a $+$particle or a $-$particle, respectively. Each $\pm$ particle moves to a neighboring empty site with the constant jump rate $C_{\pm}>0$, respectively. Two different types of neighboring particles exchange their locations with the constant rate $C_{E} \ge 0$. Also they annihilate simultaneously when they are neighboring with the constant rate $C_{A} \ge 0$, and two different types of particles are created with the constant rate $C_{C} \ge 0$ if two empty sites are neighboring.

In this paper, we consider the case where $C_{A} > 0$ or $C_{C} > 0$, so that the process has a unique conserved quantity $\sum_{x \in \T^N_d}\eta(x)$. We prove the hydrodynamic limit for the profile associated with this quantity and obtain the explicit expression of the diffusion coefficient. We classify the dynamics into three types as the case where $C_{A} > 0$ and $C_{C} > 0$ (Case 1), $C_{A} > 0$ and $C_{C} = 0$ (Case 2) and $C_{A} = 0$ and $C_{C} > 0$ (Case 3) and give proofs separately. For Cases 2 and 3, we assume the gradient condition so far. We can show that all of the hydrodynamic equations of our processes have a diagonal diffusion coefficient matrix, therefore they have unique weak solutions without the smoothness of the diffusion coefficient.

Quastel proved the hydrodynamic limit for two-colored simple exclusion process in \cite{Q}. This process is corresponding to the model with $C_+=C_->0$ and $C_A=C_C=C_E=0$, so it is not included in our cases.

The SOS dynamics describe the evolution of the integer-valued heights of interfaces on the discrete lattice. In the 1-dimensional case, the height difference of SOS dynamics and the configuration of the two-species exclusion process have one-to-one correspondence, see, e.g. \cite{CDFG}.

This paper is organized as follows: In Section 2 we introduce our model and state the main results for three types of models respectively. In Section 3, we give the proof of the main theorem for Case 1. In the proof, we give a spectral gap estimate and characterize the class of closed forms. In Sections 4 and 5, we give the proofs of the main theorems for Cases 2 and 3, respectively. In Section 6, we state the uniqueness results for nonlinear parabolic equations whose diffusion coefficient matrices are diagonal.
 
\section{Model and Main Results}
\noindent
The two-species exclusion process is a Markov process $\eta_t$ on the configuration space $\chi^d_N=\{-1,0,1\}^{\T^d_N}$, where $\torus_N^d = (\integer/N\integer)^d$ is the $d$-dimensional discrete torus. The dynamics are defined by means of an infinitesimal generator $L_N$ acting on functions $f:\chi^d_N \to \R$ as
\begin{displaymath}
(L_Nf)(\eta) = \sum_{b \in (\T_N^d)^*}L_bf(\eta),
\end{displaymath}
where $(\T_N^d)^*$ stands for the set of all directed bonds $b=(x,y)$, i.e., the ordered pairs of $x,y \in \T^d_N$ such that $|x-y|=1$ where $|x-y| = \sum_{1\le i \le d}|x_i-y_i|$ is the sum norm in $\R^d$.
Here, for each bond $b \in (\T^d_N)^*$,
\begin{equation}
 L_bf(\eta) = c_b(\eta)(\pi_b f)(\eta), \label{eq:generator}
\end{equation}
where
\begin{align*}
(\pi_b f)&(\eta) = [1_{\{\eta(x)=1,\eta(y)=0\}} + 1_{\{\eta(x)=0,\eta(y)=-1\}} + 1_{\{\eta(x)=-1,\eta(y)=1\}}](f(\eta^{x,y})-f(\eta)) \\
	&+ 1_{\{\eta(x)=1,\eta(y)=-1\}}(f(\eta^{x=0,y=0})-f(\eta)) + 1_{\{\eta(x)=0,\eta(y)=0\}}(f(\eta^{x=-1,y=1})-f(\eta)),
\end{align*} 
and
\begin{align*}
c_b(\eta) &=C_+1_{\{\eta(x)=1,\eta(y)=0\}} + C_-1_{\{\eta(x)=0,\eta(y)=-1\}}+C_E1_{\{\eta(x)=-1,\eta(y)=1\}} \\
	&+ C_A1_{\{\eta(x)=1,\eta(y)=-1\}} + C_C1_{\{\eta(x)=0,\eta(y)=0\}}.
\end{align*} 
In the above formula, $\eta^{x,y}, \eta^{x=-1,y=1}$ and $\eta^{x=0,y=0} \in \C^d_N$ stand for 
\begin{displaymath}
 \eta^{x,y}(z) = \begin{cases}
	\eta(z)  & \text{if $z \neq x,y$,} \\
	\eta(y) & \text{if $z=x$,} \\
	\eta(x)  & \text{if $z=y$,} 
\end{cases} 
\end{displaymath}
\begin{displaymath}
 \eta^{x=-1,y=1}(z) = \begin{cases}
	\eta(z)  & \text{if $z \neq x,y$,} \\
	-1 & \text{if $z=x$,} \\
	1  & \text{if $z=y$,} 
\end{cases} 
\end{displaymath}
and
\begin{displaymath}
 \eta^{x=0,y=0}(z) = \begin{cases}
	\eta(z)  & \text{if $z \neq x,y$,} \\
	0 & \text{if $z = x$,} \\
	0  & \text{if $z = y$,} \\
\end{cases}
\end{displaymath}
respectively, and we assume that $C_+$ and $C_-$ are positive constants and $C_A,C_C$ and $C_E$ are nonnegative constants.

We will use the following simplified notations
\begin{displaymath}
 \eta^{(x,y)} = \begin{cases}
	\eta^{x,y}  & \text{if $(\eta(x),\eta(y)) = (1,0)$ or $(0,-1)$ or $(-1,1)$,} \\
	\eta^{x=0,y=0} & \text{if $(\eta(x),\eta(y)) = (1,-1)$,} \\
	\eta^{x=-1,y=1}  & \text{if $(\eta(x),\eta(y)) = (0,0)$,} \\
\end{cases}
\end{displaymath}
\begin{displaymath}
\Psi^{x,y}_{a,b}(\eta)=1_{\{\eta(x)=a,\eta(y)=b\}},
\end{displaymath}
and define $r_b$ for $b \in (\T^d_N)^*$ by
\begin{align*}
r_b(\eta) &=\Psi^{x,y}_{1,0}(\eta) + \Psi^{x,y}_{0,-1}(\eta) 
	+ \Psi^{x,y}_{-1,1}(\eta) + \Psi^{x,y}_{1,-1}(\eta)+\Psi^{x,y}_{0,0}(\eta).
\end{align*} 
With these notations, $L_b$ and $\pi_b$ can be rewritten as
\begin{align*}
L_bf(\eta)&=c_b(\eta)(f(\eta^{(x,y)})-f(\eta)), \\
\pi_bf(\eta)&=r_b(\eta)(f(\eta^{(x,y)})-f(\eta)), 
\end{align*}
respectively.

The process is reversible with respect to the following one parameter family of translation invariant product measures $\nu_\rho$.
\begin{defn}
For each fixed $\rho \in [-1,1]$, let $\nu_\rho$ be a product measure on $\C_N^d$ with marginals given by
\begin{align*}
\nu_\rho\{\eta(x)=1\} & = \frac{1-\Phi(\rho)+\rho}{2},  \\
\nu_\rho\{\eta(x)=0\} &= \Phi(\rho), \\
\nu_\rho\{\eta(x)=-1\} & = \frac{1-\Phi(\rho)-\rho}{2}, 
\end{align*}
for all $x \in \T_N^d$, where
\begin{displaymath}
\Phi(\rho)=\begin{cases}
\frac{1-\sqrt{4\b+\rho^2-4\b\rho^2}}{1-4\b} & \text{if $\b \neq \frac{1}{4}$} \\
\frac{1-\rho^2}{2} & \text{if $\b=\frac{1}{4}$} \\
\end{cases}
\end{displaymath}
with $\b=C_C / C_A$.
Especially, if $C_A>0, C_C=0$, then
\begin{align*}
\nu_\rho\{\eta(x)=1\} & = \rho \vee 0,  \\
\nu_\rho\{\eta(x)=0\} &= 1-|\rho|, \\
\nu_\rho\{\eta(x)=-1\} & = |\rho \wedge 0|, 
\end{align*}
and if $C_A=0, C_C>0$, then
\begin{align*}
\nu_\rho\{\eta_x=1\} & = \frac{1+\rho}{2}, \\
\nu_\rho\{\eta_x=0\} &= 0, \\
\nu_\rho\{\eta_x=-1\} & = \frac{1-\rho}{2}. 
\end{align*}
\end{defn}

The index $\rho$ stands for the density of particles with charge, namely $E_{\nu_\rho}[\eta(0)]=\rho$. We will abuse the same notation $\nu_\rho$ for the product measures on the configuration spaces $\chi_N^d$ or $\chi^d=\{-1,0,1\}^{\Z^d}$ on the torus or on the infinite lattice. The expectation with respect to $\nu_\rho$ will be sometimes denoted by 
\begin{displaymath}
\int f(\eta)\nu_{\rho}(d\eta)= \langle f \rangle _{\rho}.
\end{displaymath}

From the definition, our model satisfies the detailed balance condition, namely, for any directed bond $b=(x,y)$, 
\begin{equation}\label{eq:dbc}
c_b(\eta)\nu_{\rho}(\eta) =c_{b^{\prime}}(\eta^{(x,y)})\nu_{\rho}(\eta^{(x,y)})
\end{equation}
holds, where $b^{\prime}=(y,x)$ is the reversed bond of $b$.

Here and after, we call $\mathfrak{f}$ a cylinder function on $\chi^d$ if $\mathfrak{f}$ depends on the configurations only through a finite set of coordinates.
For any directed bond $b=(x,y)$ and cylinder functions $f$, $g$, let us define $\D_b(\nu_{\rho};f,g)$ and $\D_b(\nu_{\rho};f)$ by 
\begin{displaymath}
\D_b(\nu_{\rho};f,g) \ := \  \langle  -(L_b+L_{b^{\prime}})f,g  \rangle _{\rho}
\end{displaymath}
and
\begin{displaymath}
\D_b(\nu_{\rho};f) \ := \D_b(\nu_{\rho};f,f),
\end{displaymath}
where $b^{\prime}=(y,x)$ and $\langle  \cdot, \cdot  \rangle _{\rho}$ stands for the inner product in $L^2(\nu_{\rho})$. The reversibility (\ref{eq:dbc}) implies
\begin{equation}
\D_b(\nu_{\rho};f,g)= \langle c_b(\pi_bf)(\pi_bg) \rangle _{\rho}.
\end{equation}

Let $\tau_x$ be the shift operator acting on the set $A \subset \Z^d$ and cylinder functions $f$ as well as configurations $\eta$ as follows:

\begin{displaymath}
\tau_xA := x + A, \quad \tau_xf(\eta)=f(\tau_x\eta), \quad (\tau_x\eta)(z):=\eta(z-x), \quad z \in \Z^d.
\end{displaymath}
\noindent
For every cylinder function $g:\chi^d \to \real$, consider the formal sum
\begin{displaymath}
\Ga_g := \sum_{x \in \Z^d}\tau_xg
\end{displaymath}
\noindent
which does not make sense but for which the gradient
\begin{displaymath}
\pi\Ga_g = (\pi_{0,e_1}\Ga_g,...,\pi_{0,e_d}\Ga_g)
\end{displaymath}
\noindent
is well defined.

We are now in a position to define the diffusion coefficient. For each $\rho \in [-1,1]$, define
\begin{displaymath}
d(\rho) =\frac{1}{\chi(\rho)}\inf_{g}\D_{0,e}(\nu_{\rho};\eta(0)+\Ga_g) 
\end{displaymath}
where $\inf_{g}$ is taken over all cylinder functions $g$ and $e$ is a unit vector of arbitrary direction. In this formula $\chi(\rho)$ stands for the so-called static compressibility which in our case is equal to
\begin{displaymath}
\chi(\rho)= \langle \eta(0)^2 \rangle _{\rho}- \langle \eta(0) \rangle _{\rho}^2=1-\Phi(\rho)-\rho^2
\end{displaymath}
Notice that $d(\rho)$ does not depend on the choice of a unit vector $e_i, \ 1 \le i \le d$. 

For a probability measure $\mu^N$ on $\C_N^d$, we denote by $\mathbb{P}_{\mu^N}$ the distribution on the path space $D(\R_+,\C_N^d)$ of the Markov process $\eta_t=\{\eta_t(x), x \in \T_N^d \}$ with generator $N^2L_N$, which is accelerated by a factor $N^2$, and the initial measure $\mu^N$. Hereafter $\mathbb{E}_{\mu^N}$ stands for the expectation with respect to $\mathbb{P}_{\mu^N}$.

With these notations our main theorems are stated as follows:
\begin{thm}\label{thm:main1}
Assume $C_A>0$ and $C_C>0$. Let $(\mu^N)_{N \ge 1}$ be a sequence of probability measures on $\C_N^d$ such that the corresponding initial density fields satisfy
\begin{displaymath}
\lim_{N \to \infty} \mu^N[|\frac{1}{N^d}\sum_{x \in \T_N^d}G(\frac{x}{N})\eta(x) - \int_{\T^d=[0,1)^d} G(u)\rho_0(u)du|>\de]=0,
\end{displaymath}
for every $\de>0$, every continuous function $G:\torus^d \to \real$ and some measurable function $\rho_0 : \torus^d \to [-1,1]$. Then, for every $t>0$,
\begin{displaymath}
\limsup_{N \to \infty} \mathbb{P}_{\mu^N}[|\frac{1}{N^d}\sum_{x \in \T_N^d}G(\frac{x}{N})\eta_t(x) - \int_{\T^d} G(u)\rho(t,u)du|>\de]=0,
\end{displaymath}
for every $\de>0$ and every continuous function $G:\torus^d \to \real$, where $\rho(t,u)$ is the unique weak solution of the following nonlinear parabolic equation:
\begin{equation}
\left\{
\begin{aligned}
\partial_{t}\rho(t,u) & = \De(\tilde{d}(\rho(t,u))) \Big( = \sum_{i=1}^d \frac{\partial}{\partial u_i} \big\{ d(\rho(t,u))\frac{\partial\rho}{\partial u_i} (t,u) \big\} \Big) \\
\rho(0,\cdot) & = \rho_0(\cdot), \label{eq:hydro}
\end{aligned}
\right.
\end{equation}
where
\begin{displaymath}
\tilde{d}(\rho)=\int_{-1}^{\rho}d(\gamma)d\gamma .
\end{displaymath}
\end{thm}
The rigorous definition of weak solutions is given in Section 6.
\begin{rem}
If we assume $C_+ + C_- -C_A-2C_E=0$, then our model turns out to be a gradient system. In this case, $d(\rho)=-\frac{\Phi^{\prime}(\rho)}{2}(C_+ - C_-)+\frac{1}{2}(C_+ + C_-)$ holds. In particular, we can compute the diffusion coefficient $d(\rho)$ explicitly from the concrete values of $C_+, C_-$ and $\b$.
\end{rem}

\begin{rem}
Generalized exclusion process with $\kappa=2$ is corresponding to our model with $C_+=C_-=C_A=C_C=1$ and $C_E=0$.
\end{rem}

\begin{thm}\label{thm:main2}
Assume $C_A>0$, $C_C=0$ and the gradient condition $C_+ + C_- -C_A-2C_E=0$. Let $(\mu^N)_{N \ge 1}$ satisfy the same assumption as in Theorem \ref{thm:main1}. Then, for every $t>0$,
\begin{displaymath}
\limsup_{N \to \infty} \mathbb{P}_{\mu^N}[|\frac{1}{N^d}\sum_{x \in \T_N^d}G(\frac{x}{N})\eta_t(x) - \int_{\T^d} G(u)\rho(t,u)du|>\de]=0,
\end{displaymath}
for every $\de>0$ and every continuous function $G:\torus^d \to \real$, where $\rho(t,u)$ is the unique weak solution of the following nonlinear parabolic equation:
\begin{equation}
\left\{
\begin{aligned}
\partial_{t}\rho(t,u) & = \De(P(\rho(t,u))) \Big( = \sum_{i=1}^d \frac{\partial^2}{\partial u_i^2} P(\rho(t,u)) \Big) \\
\rho(0,\cdot) & = \rho_0(\cdot), \label{eq:hydro2}
\end{aligned}
\right.
\end{equation}
where
\begin{equation}
P(\rho)=C_+\rho1_{\{\rho>0\}}-C_-\rho1_{\{\rho<0\}}. \label{eq:diffusion2}
\end{equation}
\end{thm}

\begin{rem}
Equation (\ref{eq:hydro2}) is the weak or enthalpy formulation of the following two-phases Stefan problem:
\begin{displaymath}
\left\{
\begin{aligned}
\partial_{t}\rho(t,u) & = C_+ \De \rho(t,u) \quad \text{on} \quad \mathcal{L}(t)=\{\rho(t,u)>0\} \\
\partial_{t}\rho(t,u) & = C_- \De \rho(t,u) \quad \text{on} \quad \mathcal{S}(t)=\{\rho(t,u) <0\} \\
0 & = \bold{n} \cdot \Big( C_+ \nabla(\rho(t,u) \vee 0) - C_- \nabla(\rho(t,u) \wedge 0) \Big) \\
 & \quad \quad \quad \text{on} \quad \Sigma(t)=\{\rho(t,u)=0\}\\
\rho(0,\cdot) & = \rho_0(\cdot)
\end{aligned}
\right.
\end{displaymath}
where $\bold{n}$ denotes the unit normal vector on $\Sigma(t)$ directed to $\mathcal{L}(t)$ and $\nabla(\rho(t,u) \vee 0)$ (respectively $\nabla(\rho(t,u) \wedge 0)$) is the limit of the gradient of $\rho \vee 0$ (respectively $\rho \wedge 0$) at $u \in \Sigma(t)$ when approached from $\mathcal{L}(t)$ (respectively $\mathcal{S}(t)$), see \cite{F}.
\end{rem}

\begin{thm}\label{thm:main3}
Assume $C_A=0$, $C_C>0$ and the gradient condition $C_+ + C_- -2C_E=0$. Let $(\mu^N)_{N \ge 1}$ satisfy the same assumption as in Theorem \ref{thm:main1}. Then, for every $t>0$,
\begin{displaymath}
\limsup_{N \to \infty} \mathbb{P}_{\mu^N}[|\frac{1}{N^d}\sum_{x \in \T_N^d}G(\frac{x}{N})\eta_t(x) - \int_{\T^d} G(u)\rho(t,u)du|>\de]=0,
\end{displaymath}
for every $\de>0$ and every continuous function $G:\torus^d \to \real$, where $\rho(t,u)$ is the unique weak solution of the heat equation:
\begin{equation}
\left\{
\begin{aligned}
\partial_{t}\rho(t,u) & = C_E\De \rho(t,u) \Big( = C_E \sum_{i=1}^d \frac{\partial^2}{\partial u_i^2} \rho(t,u) \Big) \\
\rho(0,\cdot) & = \rho_0(\cdot). \label{eq:hydro3}
\end{aligned}
\right.
\end{equation}
\end{thm}

\section{Proof of Theorem \ref{thm:main1}}
In this section, we consider Case 1, namely the dynamics with $C_A>0$ and $C_C>0$. For this case, we do not assume the gradient condition, so the system is, in general, nongradient. The strategy of the proof is essentially the same as given for the generalized exclusion process in \cite{KL}. The main step is obtaining the estimate of the spectral gap and the characterization of the closed forms, which are presented in the Subsections 3.4 and 3.5. To guarantee the uniqueness of the weak solution, we show that the diffusion coefficient matrix is diagonal. The method of the proof is also available for the large class of symmetric processes including generalized exclusion process.

\subsection{The Macroscopic Equation}
We start with considering a class of martingales associated with the empirical measure. We take $T>0$ arbitrarily and fix it in the rest of this section. For each smooth function $H:\T^d \to \R$, let $M^{H,N}(t)=M^{H}(t)$ be the martingale defined by
\begin{displaymath}
M^{H}(t)= \langle \pi^N_t, H \rangle  -  \langle \pi^N_0, H \rangle  - \int^t_0 N^2 L_N \langle \pi^N_s, H \rangle ds,
\end{displaymath}
where $\pi^N_t$ stands for the empirical measure associated with $\eta_t$, namely
\begin{equation}\label{eq:empirical}
\pi^N_t(du)=\frac{1}{N^d} \sum_{x \in \T_N^d} \eta_{t}(x) \de_{\frac{x}{N}}(du),\quad  0 \le t \le T, \quad u \in \T^d,
\end{equation}
and $ \langle \pi^N_t, f \rangle $ stands for the integration of $f$ with respect to $\pi^N_t$.

A simple computation shows that the expected value of the quadratic variation of $M^H(t)$ vanishes as $N \uparrow \infty$, and therefore by Doob's inequality, for every $\de>0$, we have
\begin{displaymath}
\lim_{N \to \infty}\mathbb{P}_{\mu^N}[\sup_{0 \le t \le T}|M^H(t)| \ge \de]=0.
\end{displaymath}

A spatial summation by parts permits to rewrite the martingale $M^{H}(t)$ as
\begin{align}
M^{H}(t) &= \langle \pi^N_t, H \rangle  -  \langle \pi^N_0, H \rangle  \nonumber \\
	& - \sum_{i=1}^d\int^t_0 N^{1-d}\sum_{x \in \T^d_N}(\partial_{u_i}^NH)(\frac{x}{N})\tau_xW_{0,e_i}(\eta_s)ds, \label{eq:martingale}
\end{align}
where $W_{0,e_i}$ represents the instantaneous current from $0$ to $e_i$:
\begin{align*}
W_{0,e_i}(\eta) &= C_+ (\Psi^{0,e_i}_{1,0}(\eta)-\Psi^{0,e_i}_{0,1}(\eta) ) 		+ (C_A+2C_E) (\Psi^{0,e_i}_{1,-1}(\eta)-\Psi^{0,e_i}_{-1,1}(\eta) ) \\
	&+ C_- (\Psi^{0,e_i}_{0,-1}(\eta)-\Psi^{0,e_i}_{-1,0}(\eta))
\end{align*}
and $\partial_{u_i}^NH$ represents the discrete derivative of $H$ in the $i$-th direction:
\begin{displaymath}
(\partial_{u_i}^NH)(\frac{x}{N})=N[H(\frac{x+e_i}{N})-H(\frac{x}{N})].
\end{displaymath}

Next we show that the current $W_{0,e_i}$ can be decomposed into a linear combination of the gradients $\{\eta(e_j)-\eta(0), 1 \le j \le d\}$ and a function in the range of the generator $L_N$: $W_{0,e_i} + \sum_{j=1}^d D_{i,j}(\rho)[\eta(e_j) - \eta(0)]=L_N\mathfrak{f}$ for a certain cylinder function $\mathfrak{f}$ and some matrix $D_{i,j}(\rho)$ that depends on the density, see Theorem \ref{thm:keythm} and Corollary \ref{cor:keycor} for more precise statement. 

Denote by $\{D_{i,j}(\rho),1 \le i,j \le d\}$ the unique symmetric matrix such that 
\begin{displaymath}
a^{*}D(\rho)a=\frac{1}{\chi(\rho)}\inf_g \Big\{\sum_{i=1}^d \D_{0,e_i}(\nu_{\rho},a_i\eta(0)+\Ga_g)\Big\},
\end{displaymath}
for every vector $a$ in $\R^d$ where $\inf_{g}$ is taken over all cylinder functions $g$.

For positive integers $l, N,$ a function $H$ in $C^2(\T^d)$ and a cylinder function $\mathfrak{f}$ on $\chi^d$, let
\begin{displaymath}
X^{\mathfrak{f},i}_{N,l}(H,\eta)=N^{1-d}\sum_{x \in \T^d_N}H(\frac{x}{N})\tau_xV^{\mathfrak{f},l}_i(\eta),
\end{displaymath}
where
\begin{displaymath}
V^{\mathfrak{f},l}_i(\eta)=W_{0,e_i}(\eta) + \sum_{j=1}^d D_{i,j}(\eta^l(0))[\eta^l(e_j) - \eta^l(0)]-L_N\mathfrak{f}(\eta),
\end{displaymath}
and
\begin{displaymath}
\eta^l(x)=\frac{1}{(2l+1)^d}\sum_{|y-x| \le l}\eta(y), \quad x \in \T_N^d.
\end{displaymath}
\begin{thm} \label{thm:keythm}
Fix $\rho \in (-1,1)$ arbitrarily. Then, for every function $H$ in $C^2(\T^d)$ and $1 \le i \le d$, we have
\begin{displaymath}
\inf_{\mathfrak{f} \in \Cy}\limsup_{\e \to 0}\limsup_{N \to \infty}\frac{1}{N^d}\log\mathbb{E}_{\nu_{\rho}^N}[\exp\{N^d|\int^T_0 X^{\mathfrak{f},i}_{N,\e N}(H,\eta_s)ds|\}] = 0,
\end{displaymath}
where $\Cy$ stands for the set of cylinder functions on $\chi^d$.
\end{thm}

The proof of Theorem \ref{thm:keythm} is postponed to the next subsection. This theorem implies the following corollary. For a positive integer $l$ and a function $H$ in $C^2(\T^d)$, let
\begin{displaymath}
Y^i_{N,l}(H, \eta)= N^{1-d}\sum_{x \in \T^d_N}H(\frac{x}{N})\{W_{x,x+e_i}(\eta) + \sum_{j=1}^d D_{i,j}(\eta^l(x))[\eta^l(x+e_j) - \eta^l(x)]\}.
\end{displaymath}

\begin{cor}\label{cor:keycor}
For every function $H$ in $C^2(\T^d)$ and $1 \le i \le d$,
\begin{displaymath}
\limsup_{\e \to 0}\limsup_{N \to \infty}\mathbb{E}_{\mu^N}[|\int^T_0 Y^i_{N,\e N}(H,\eta_s)ds|] = 0.
\end{displaymath}
\end{cor}

To prove this corollary, we can use the method in \cite{KL} straightforwardly. In particular, the $L_N\mathfrak{f}$ term is negligible. We have now all elements to prove the hydrodynamic behavior of our nongradient system.
\begin{proof}[Proof of Theorem \ref{thm:main1}]
Recall that the empirical measure $\pi^N_t$ is defined by (\ref{eq:empirical}). Denote by $Q_{\mu^N}$ the distribution on the path space $D([0,T], \mathcal{M}(\T^d))$ of the process $\pi^N_t$ where $\mathcal{M}(\T^d)$ stands for the space of signed measures on $\T^d$ endowed with the weak topology.

Following the same argument as for the generalized exclusion process in \cite{KL} it is easy to prove that the sequence $\{Q_{\mu^N}, N \ge 1 \}$ is weakly relatively compact and that every limit points $Q^{*}$ is concentrated on absolutely continuous paths $\pi_t(du)=\pi(t,u)du$ with density bounded by 1 and -1 from above and below respectively: $-1 \le \pi(t,u) \le 1$.

From Theorem \ref{thm:uniqueness} stated below, there exists at most one weak solution of (\ref{eq:hydro}). Therefore, to conclude the proof of the theorem, it remains to show that all limit points of the sequence $\{Q_{\mu^N}, N \ge 1\}$ are concentrated on absolutely continuous trajectories $\pi(t,du)=\pi(t,u)du$ whose densities are weak solutions of the equation (\ref{eq:hydro}). 

Fix a smooth function $H:\T^d \to \R$ and recall the definition of the martingale $M^{H}(t)$. Applying Corollary \ref{cor:keycor} to the last integral term in the formula (\ref{eq:martingale}) of $M^{H}(t)$, we obtain that for every $\de>0$,
\begin{align*}
\limsup_{\e \to 0}&\limsup_{N \to \infty}\mathbb{P}_{\mu^N}[| \langle \pi^N_T, H \rangle  -  \langle \pi^N_0, H \rangle \\
	& + \sum_{i,j=1}^d\int^T_0 N^{1-d}\sum_{x \in \T^d_N}(\partial_{u_i}^NH)(\frac{x}{N})\tau_xV_{i,j,\e N}(\eta_s)ds| > \de] = 0,
\end{align*} 
where
\begin{displaymath}
V_{i,j,\e N}(\eta)=D_{i,j}(\eta^{\e N}(0))[\eta^{\e N}(e_j)-\eta^{\e N}(0)].
\end{displaymath}
Denote by $\tilde{D}_{i,j}$ the integral of $D_{i,j}: \tilde{D}_{i,j}(\rho)=\int_{-1}^{\rho}D_{i,j}(\gamma)d\gamma$. Since $H$ is smooth and $D_{i,j}$ is continuous by Theorem \ref{thm:conti} stated below, with the help of Taylor's expansion and a spatial summation by parts, we have
\begin{align*}
\limsup_{\e \to 0}&\limsup_{N \to \infty}\mathbb{P}_{\mu^N}[| \langle \pi^N_T, H \rangle  -  \langle \pi^N_0, H \rangle \\
	& - \sum_{i,j=1}^d\int^T_0 N^{-d}\sum_{x \in \T^d_N}(\partial^2_{u_i,u_j}H)(\frac{x}{N})\tilde{D}_{i,j}(\eta^{\e N}_s(x))ds| > \de] = 0.
\end{align*}
Therefore, for every limit point $Q^{*}$ of the sequence $Q_{\mu^N}$,
\begin{align*}
\limsup_{\e \to 0}& Q^{*}[| \langle \pi_T, H \rangle  -  \langle \pi_0, H \rangle  \\
	& - \sum_{i,j=1}^d\int^T_0 ds \int_{\T^d} du (\partial^2_{u_i,u_j}H)(u)\tilde{D}_{i,j}((\pi_s * \iota_{\e})(u))| > \de] = 0
\end{align*}
where
\begin{displaymath}
\iota_{\e}(\cdot):=(2\e)^{-d}1_{[-\e,\e]^d}(\cdot)
\end{displaymath}
and $*$ represents the convolution. Since each limit point $Q^{*}$ is concentrated on absolutely continuous paths $\pi_t=\pi(t,u)du$ with $-1 \le \pi(t,u) \le 1$, for each fixed $0 \le s \le T$, $(\pi_s * \iota_{\e})(u)$ converges to $\pi(s,u)$ for almost $u$ in $\T^d$ as $\e \downarrow 0$. From this remark and the continuity of $\{ \tilde{D}_{i,j}, 1 \le i,j \le d \}$, we obtain that
\begin{align*}
Q^{*}[\Big| \langle \pi_T, H \rangle  -  \langle \pi_0, H \rangle - \sum_{i,j=1}^d\int^T_0 ds \int_{\T^d} du (\partial^2_{u_i,u_j}H)(u)\tilde{D}_{i,j}(\pi(s,u))\Big| > \de] = 0
\end{align*}
for all $H$ in $C^2(\T^d)$. The fact $D(\rho)=d(\rho)I$, namely $\tilde{D}(\rho)=\tilde{d}(\rho)I$ proved in Theorem \ref{thm:diagonal} permits to rewrite the last expression as
\begin{align*}
Q^{*}[\langle \pi_T, H \rangle  = \langle \pi_0, H \rangle + \int^T_0 ds \int_{\T^d} du \De H(u)\tilde{d}(\pi(s,u)) ] = 1.
\end{align*}
Denote by $\{t_n\}_{n \in \N}$ a dense subset of $[0,T]$ and repeat the same argument as we have done up to this point for any fixed $t_n$, then 
\begin{align*}
Q^{*}[\langle \pi_{t_n}, H \rangle  = \langle \pi_0, H \rangle + \int^{t_n}_0 ds \int_{\T^d} du \De H(u)\tilde{d}(\pi(s,u)) \quad \text{for every} \quad n \in \N] = 1.
\end{align*}
Since $Q^*$ is the probability measure on $D$ space and $\tilde{d}$ is bounded function on $[-1,1]$ and $Q^{*}$ is concentrated on paths $\pi_t=\pi(t,u)du$ with $-1 \le \pi(t,u) \le 1$, $Q^{*}$ is concentrated on the weak solution of (\ref{eq:hydro}) which concludes the proof of the theorem.
\end{proof}

\subsection{Central Limit Theorem Variances}
To state the main theorem of this subsection, first we introduce some notation. For a fixed positive integer $l$ we denote by $\La_l$ a cube in $\Z^d$ of side-length $2l+1$ centered at the origin: $\La_l:=\{-l, -l+1,..., l-1,l\}^d$. We denote the set of cylinder functions on $\chi^d$ by $\Cy$. For $\Psi$ in $\Cy$, denote by $\La_{\Psi}$ the smallest $d$-dimensional rectangle that contains the support of $\Psi$ and by $s_{\Psi}$ the smallest positive integer $s$ such that $\La_{\Psi} \subset \La_s$. Let $\Cy_0$ be the space of cylinder functions with mean zero with respect to all canonical invariant measures:
\begin{displaymath}
\Cy_0 = \{\ g \in \Cy \ ;  \langle g \rangle _{\La_g,K} = 0 \ \text{for all} \ -|\La_g| \le K \le |\La_g| \ \}.
\end{displaymath}
Here, for a finite subset $\La$ of $\Z^d$, we denote by $|\La|$ the cardinality of $\La$ and by $ \langle  \cdot  \rangle _{\La,K}$ the expectation with respect to the canonical measure $\nu_{\La,K}:=\nu_{\a}( \ \cdot \ | \sum_{x \in \La} \eta(x)=K)$ for $-|\La| \le K \le |\La|$ which is indeed independent of the choice of $\a$. For a rectangle $\La$ and a canonical measure $\nu_{\La,K}$, denote by $ \langle  \cdot, \cdot  \rangle _{\La,K}$ (resp.$ \langle  \cdot, \cdot  \rangle _{\a}$) the inner product in $L^2(\nu_{\La,K})$ (resp. $L^2(\nu_{\a})$).

It is known that to conclude the proof of Theorem \ref{thm:keythm} it is enough to show that
\begin{equation}
\inf_{\mathfrak{f} \in \Cy}\lim_{l \to \infty}\sup_K (2l)^{d} \langle (-L_{\La_l})^{-1}\tilde{V}^{\mathfrak{f},l}_{i}, \tilde{V}^{\mathfrak{f},l}_{i} \rangle _{l,K}=0 \label{eq:variance}
\end{equation}
where 
\begin{align*}
\tilde{V}^{\mathfrak{f},l}_i(\eta)&=(2l^{\prime}+1)^{-d}\sum_{|y| \le l^{\prime}}\tau_yW_{0,e_i}(\eta)  \\
&+ \sum_{j=1}^d D_{i,j}(\eta^l(0))[\eta^{l^{\prime}}(e_j) - \eta^{l^{\prime}}(0)]-(2l_{\mathfrak{f}}+1)^{-d}\sum_{y \in \La_{l_{\mathfrak{f}}}}(\tau_yL_N\mathfrak{f})(\eta),
\end{align*}
$l^{\prime}=l-1$ and $l_{\mathfrak{f}}=l-s_{\mathfrak{f}}-1$ so that $\tau_yL_N\mathfrak{f}$ is $\mathcal{F}_{\La_l}$-measurable for every $y$ in $\La_{l_{\mathfrak{f}}}$. This follows from Theorem \ref{thm:converge} and Corollary \ref{cor:uniconti} below.

For the beginning of the proof we obtain a variational formula for this variance. We start with introducing a semi-norm on $\Cy_0$, which is closely related to the central limit theorem variance. For $1 \le k \le d$ denote $\mathcal{U}^k=(\mathcal{U}_1^k,...,\mathcal{U}_d^k)$ the $d$-dimensional cylinder function with coordinates defined by
\begin{displaymath}
(\mathcal{U}^k)_i(\eta) = \de_{i,k}\nabla_{0,e_k}\eta(0) \quad \text{for all} \quad 1 \le i \le d.
\end{displaymath}
Here $\de_{i,j}$ stands for the delta of Kronecker. For  cylinder functions $g$, $h$ in $\Cy_0$ and $1 \le i \le d$, let
\begin{displaymath}
\ll g,h \gg_{\rho,0}= \sum_{x \in \Z^d} \langle g, \tau_xh \rangle _{\rho} \quad \text{and} \quad \ll g \gg_{\rho,j}=\sum_{x \in \Z^d}x_j \langle g,\eta(x) \rangle _{\rho},
\end{displaymath}
where $x_j$ stands for the $j$-th coordinate of $x \in \Z^d$. Both $\ll g,h \gg_{\rho,0}$ and $\ll g \gg_{\rho,j}$ are well defined because $g$ and $h$ belong to $\Cy_0$ and therefore all but a finite number of terms vanish. For $h$ in $\Cy_0$, define the semi-norm $\ll h \gg_{\rho}^{\frac{1}{2}}$ by 
\begin{align}
\ll &h \gg_{\rho}  \\ 
&=\sup_{g \in \Cy_0, a \in \R^d}\{2\ll g,h \gg_{\rho,0}+2\sum_{i=1}^d a_i \ll h \gg_{\rho,i} - \sum_{i=1}^d \langle \big( \sum_{j=1}^d a_j(\mathcal{U}^j)_i + \nabla_{0,e_i}\Gamma_g \big)^2 \rangle _{\rho}\} \nonumber \\
&=\sup_{g \in \Cy_0, a \in \R^d}\{2\ll g,h \gg_{\rho,0}+2\sum_{i=1}^d a_i \ll h \gg_{\rho,i} - \sum_{i=1}^d \D_{0,e_i}(\nu_{\rho};a_i\eta(0)+\Gamma_g) \}, \nonumber
\end{align}
where $a=(a_i)_{i=1}^d$.

We investigate in the next section several properties of the semi-norm $\ll \cdot \gg_{\rho}^{\frac{1}{2}}$, while in this section we prove that the variance
\begin{displaymath}
(2l)^{-d} \langle (-L_{\La_l})^{-1}\sum_{|x| \le l_{\psi}}\tau_x\psi, \sum_{|x| \le l_{\psi}}\tau_x\psi \rangle _{l,K_l}
\end{displaymath}
of any cylinder function $\psi$ in $\Cy_0$ converges to $\ll \psi \gg_{\rho}$, as $ l \uparrow \infty$ and $\frac{K_l}{(2l)^d} \to \rho$. Here $l_{\psi}$ stands for $l-s_{\psi}$ so that the support of $\tau_x\psi$ is included in $\La_l$ for every  $x \le l_{\psi}$. By elementary computations relying on an adequate change of variables, the norm $\ll \cdot \gg_{\rho}$ may be rewritten as
\begin{align*}
\ll h \gg_{\rho} &=\sup_{g \in \Cy_0, a \in \R^d}\{2\ll g,h \gg_{\rho,0}+2\sum_{i=1}^d a_i \ll h \gg_{\rho,i} \\
	& +2 \sum_{i=1}^d a_i \ll W_{0,e_i},g \gg_{\rho,0} - \|a\|^2 \langle (\nabla_{0,e_1}\eta(0))^2 \rangle _{\rho} - \langle \|\nabla\Gamma_g \|^2 \rangle _{\rho}\}.
\end{align*}

We are now in a position to state the main result of this section.
\begin{prop}\label{prop:variance}
Consider a cylinder function $\psi$ in $\Cy_0$ and a sequence of integers $K_l$ such that $-(2l+1)^d \le K_l \le (2l+1)^d$ and $\lim_{\l \to \infty} \frac{K_l}{(2l)^d}=\rho$. Then,
\begin{displaymath}
\lim_{l \to \infty} (2l)^{-d} \langle (-L_{\La_l})^{-1}\sum_{|x| \le l_{\psi}}\tau_x\psi, \sum_{|x| \le l_{\psi}}\tau_x\psi \rangle _{l,K_l} = \ll \psi \gg_{\rho}.
\end{displaymath}
\end{prop}

Once Theorem \ref{thm:closedform}, which is stated below, is established the proof of Proposition \ref{prop:variance} is the same as that of Theorem 7.4.1 of \cite{KL} since the proof does not depend on the specific form of $\D_b$.

We conclude this section proving that for each $\psi$ in $\Cy_0$ the function $\ll \psi \gg : [-1,1] \to \R_+$ that associates to each density $\rho$ the value $\ll \psi \gg_{\rho}$ is continuous and that the convergence of the finite volume variances to $\ll \cdot \gg_{\rho} $ is uniform on $[-1,1]$. For each $l$ in $\N$ and $ -(2l+1)^d \le K \le (2l+1)^d$, denote by $V_l^{\psi}\big(\frac{K}{(2l+1)^d}\big)$ the variance of $(2l+1)^{-d}\sum_{|x| \le l_{\psi}}\tau_x\psi$ with respect to $\nu_{l,K}$:
\begin{displaymath}
V_l^{\psi}\Big(\frac{K}{(2l+1)^d}\Big)=(2l)^{-d} \langle (-L_{\La_l})^{-1}\sum_{|x| \le l_{\psi}}\tau_x\psi, \sum_{|x| \le l_{\psi}}\tau_x\psi \rangle _{l,K} 
\end{displaymath}

We may interpolate linearly to extend the definition of $V_l^{\psi}$ to the all interval $[-1,1]$. With this definition $V_l^{\psi}$ is continuous. Proposition \ref{prop:variance} asserts that $V_l^{\psi}$ converges, as $l \uparrow \infty$, to $\ll \psi \gg_{\rho}$, for any sequence $K_l$ such that $\frac{K_l}{(2l+1)^d} \to \rho$. In particular, $\lim_{l \to \infty}V_l^{\psi}(\rho_l)=\ll \psi \gg_{\rho}$ for any sequence $\rho_l \to \rho$. This implies that $\ll \psi \gg_{\rho}$ is continuous and that $V_l^{\psi}(\cdot)$ converges uniformly to $\ll \psi \gg_{\cdot}$ as $l \uparrow \infty$. We have thus proved the following theorem.

\begin{thm}\label{thm:converge}
For each fixed $h$ in $\Cy_0$, $\ll h \gg_{\rho}$ is continuous as a function of the density $\rho$ on $[-1,1]$. Moreover, the variance
\begin{displaymath}
(2l)^{-d} \langle (-L_{\La_l})^{-1}\sum_{|x| \le l_h}\tau_xh, \sum_{|x| \le l_h}\tau_xh \rangle _{l,K_l} 
\end{displaymath}
converges uniformly to $\ll h \gg_{\rho}$ as $l \uparrow \infty$ and $\frac{K_l}{(2l+1)^d} \to \rho$. In particular,
\begin{displaymath}
\lim_{l \to \infty} \sup_{-(2l+1)^d \le K \le (2l+1)^d} (2l)^{-d} \langle (-L_{\La_l})^{-1}\sum_{|x| \le l_h}\tau_xh, \sum_{|x| \le l_h}\tau_xh \rangle _{l,K}=\sup_{-1 \le \rho \le 1}\ll h \gg_{\rho}.
\end{displaymath}
\end{thm}

\subsection{The Diffusion Coefficient}

We investigate here the main properties of the semi norm $\ll \cdot \gg_{\rho}$ introduced in the previous section. We first define from $\ll \cdot \gg_{\rho}$ a semi-inner product on $\Cy_0$ through polarization:
\begin{equation}
\ll g,h \gg_{\rho} = \frac{1}{4}\{\ll g+h \gg_{\rho} - \ll g-h \gg_{\rho} \}. \label{eq:semiin}
\end{equation}

It is easy to check that (\ref{eq:semiin}) defines a semi-inner product on $\Cy_0$. Denote by $\mathcal{N}_{\rho}$ the kernel of the semi-norm $\ll \cdot \gg_{\rho}^{\frac{1}{2}}$ on $\Cy_0$. Since $\ll \cdot \gg_{\rho}$ is a semi-inner product on $\Cy_0$, the completion of $\Cy_0|_{\mathcal{N}_{\rho}}$, denoted by $\mathcal{H}_{\rho}$, is a Hilbert space.

Simple computations show that the linear space generated by the currents $\{W_{0,e_i},\ 1 \le i \le d\}$ and $L\Cy_0 = \{ Lg; \ g \in \Cy_0 \}$ are subsets of $\Cy_0$. The first main result of this section consists in showing that $\mathcal{H}_{\rho}$ is the completion of $L\Cy_0|_{\mathcal{N}_{\rho}}+\{W_{0,e_i}, 1 \le i \le d\}$, in other words, that all elements of $\mathcal{H}_{\rho}$ can be approximated by $\sum_{1 \le i \le d}a_iW_{0,e_i} + Lg$ for some $a$ in $\R^d$ and $g$ in $\Cy_0$. To prove this result we derive two elementary identities:
\begin{equation}
\ll h,Lg \gg_{\rho} = - \ll h,g \gg_{\rho,0} \ \text{and} \quad \ll h,W_{0,e_i} \gg_{\rho}=-\ll h \gg_{\rho,i} \label{eq:keyid}
\end{equation}
for all $h,g$ in $\Cy_0$ and $1 \le i \le d$.

By Proposition \ref{prop:variance} and (\ref{eq:semiin}), the semi-inner product $\ll h,g \gg_{\rho}$ is the limit of the covariance $(2l)^{-d} \langle (-L_{\La_l})^{-1}\sum_{|x| \le l_g}\tau_xg, \sum_{|x| \le l_h}\tau_xh \rangle _{l,K_l}$ as $l \uparrow \infty$ and $\frac{K_l}{(2l)^d} \to \rho$. In particular, if $g=Lg_0$, for some cylinder function $g_0$, the inverse of the generator cancels with the generator. Therefore, $\ll h,Lg_0 \gg_{\rho}$ is equal to
\begin{displaymath}
-\lim_{l \to \infty} (2l)^{-d} \langle \sum_{|x| \le l_{g_0}}\tau_xg_0, \sum_{|x| \le l_h}\tau_xh \rangle _{l,K_l}= \ll g_0,h \gg_{\rho,0}.
\end{displaymath}
The second identity is proved in a similar way.

It follows from the first identity of (\ref{eq:keyid}) that the gradients $\{\eta(e_i)-\eta(0), 1 \le i \le d \}$ are orthogonal to the space $L\Cy_0$, while the second identity permits to compute inner product of cylinder functions with the current:
\begin{align}
\ll \eta(e_i)-\eta(0), Lh \gg_{\rho} &=0,  \label{eq:cal1}\\
\ll \eta(e_i)-\eta(0), W_{0,e_j} \gg_{\rho}& = -\chi(\rho)\delta_{i,j}, \label{eq:cal2}
\end{align}
and 
\begin{equation}
 \ll W_{0,e_i}, W_{0,e_j} \gg_{\rho} =   \langle (\nabla_{0,e_1}\eta(0))^2 \rangle _{\rho}\delta_{i,j}. \label{eq:cal3} 
\end{equation}
for all $1 \le i,j \le d$ and $h \in \Cy_0$.
In this formula $\chi(\rho)$ stands for the static compressibility and is equal to $ \langle \eta(0)^2 \rangle _{\rho} - \langle \eta(0) \rangle _{\rho}^2$. Furthermore,
\begin{equation}
\ll \sum_{j=1}^d a_j W_{0,e_j} +Lg \gg_{\rho} = \sum_{i=1}^d  \langle  \{ \nabla_{0,e_i}(a_i \eta(0) + \Gamma_g )\}^2 \rangle _{\rho} \label{eq:current}
\end{equation}
for $a$ in $\R^d$ and $g$ in $\Cy_0$. In particular, the variational formula for $\ll h \gg_{\rho}$ writes
\begin{equation}
\ll h \gg_{\rho} =\sup_{g \in \Cy_0, a \in \R^d}\{ -2 \ll h, \sum_{i=1}^d a_i W_{0,e_i}+Lg \gg_{\rho} - \ll \sum_{i=1}^d a_i W_{0,e_i} +Lg \gg_{\rho} \}. \label{eq:variational}
\end{equation}

\begin{prop}
Recall that we denote by $L\Cy_0$ the space $\{ Lg; \ g \in \Cy_0 \}$. Then, for each $-1 \le \rho \le 1$, we have
\begin{displaymath}
\mathcal{H}_{\rho} = \overline{L\Cy_0}|_{\mathcal{N}_{\rho}} \oplus \{W_{0,e_i}, 1 \le i \le d\}.
\end{displaymath}
\end{prop}

\begin{proof}
We can apply the proof of Proposition 7.5.2 in \cite{KL} straightforwardly.
\end{proof}

\begin{cor}\label{cor:matrix}
For each $g \in \Cy_0$, there exists a unique vector $a \in \R^d$ such that
\begin{displaymath}
g-\sum_{j=1}^d a_j W_{0,e_j}  \in \overline{L\Cy_0} \quad \text{in} \quad \mathcal{H}_{\rho}.
\end{displaymath}
\end{cor}

We now start to describe the diffusion coefficient $D$ of the hydrodynamic equation. From Corollary \ref{cor:matrix}, there exists a matrix $ \{ Q_{i,j}, 1 \le i,j \le d\}$  such that
\begin{equation}
\eta(e_i) - \eta(0) + \sum_{j=1}^d Q_{i,j} W_{0,e_j}  \in \overline{L\Cy_0} \quad \text{in} \quad \mathcal{H}_{\rho}. \label{eq:matrixq}
\end{equation}

Notice that the matrix $Q=Q(\rho)$ depends on the density $\rho$ because the inner product depends on $\rho$. It is easily shown that $Q$ is symmetric and strictly positive.

Denote by $D = D(\rho)$ the inverse of $Q$, which is also symmetric and strictly positive. We will see below that $D(\rho)$ is the diffusion coefficient of the hydrodynamic equation (\ref{eq:hydro}). Since $D$ is the inverse of $Q$, we have that
\begin{displaymath}
W_{0,e_i} + \sum_{j=1}^d D_{i,j}[\eta(e_j) - \eta(0)]   \in \overline{L\Cy_0} \quad \text{in} \quad \mathcal{H}_{\rho}.
\end{displaymath}
for $1 \le i \le d$. This relation provides a variational characterization of the diffusion coefficient $D$. Indeed, for all vectors $a \in \R^d$,
\begin{equation}
\inf_{g \in \Cy_0}\{ \ll \sum_{i=1}^d a_i W_{0,e_i} + \sum_{i,j=1}^d a_i D_{i,j}[\eta(e_j) - \eta(0)] -Lg \gg_{\rho} \} =0 .
\end{equation}

Since gradients are orthogonal to the space $L\Cy_0$,
\begin{displaymath}
\ll \eta(e_j)-\eta(0), W_{0,e_i} \gg_{\rho} = -\chi(\rho)\delta_{i,j}, 
\end{displaymath}
and  
\begin{displaymath}
\ll \eta(e_j)-\eta(0), \eta(e_k)-\eta(0) \gg_{\rho} = \chi(\rho)Q_{j,k} = \chi(\rho)[D^{-1}]_{j,k} ,
\end{displaymath}
the last identity reduces to
\begin{displaymath}
\inf_{g \in \Cy_0}\{ -\chi(\rho)a^{*}Da + \ll \sum_{i=1}^d a_i W_{0,e_i} -Lg \gg_{\rho} \} =0 ,
\end{displaymath}
where $a^{*}$ stands for the transposition of $a$. We have thus obtained a variational formula for $D(\rho)$.

\begin{thm}\label{thm:coe}
The diffusion coefficient $D(\rho)$ is such that
\begin{equation}
a^{*}Da = \frac{1}{\chi(\rho)} \inf_{g \in \Cy_0} \ll \sum_{i=1}^d a_i W_{0,e_i} -Lg \gg_{\rho} = \frac{1}{\chi(\rho)} \inf_{g \in \Cy_0} \sum_{i=1}^d  \langle \{ ( \nabla_{0,e_i}(a_i \eta(0)- \Gamma_g )\}^2 \rangle _{\rho} \label{eq:inf}
\end{equation}
for all $a \in \R^d$.
\end{thm}

The second identity follows from equation (\ref{eq:current}). Moreover, this formula determines the matrix $D$ since $D$ is symmetric.

It is now easy to prove the diffusion coefficient is continuous including at the boundary of $[-1,1]$. From the explicit formulas for $\chi(\rho)$, $\langle (\nabla_{0,e_1}\eta(0))^2 \rangle_{\rho}$ and $\langle \Psi_{0,e_1}^2 \rangle_{\rho}$, we have that $D(\rho)$ converges to $C_+I$ as $\rho \uparrow 1$ and $C_-I$ as $\rho \downarrow -1$. 
\begin{thm}\label{thm:conti}
The diffusion coefficient $D(\rho)$ is continuous on $[-1,1]$. Moreover it converges to $C_+I$ as $\rho \uparrow 1$ and $C_-I$ as $\rho \downarrow -1$. 
\end{thm}

From the continuity of the diffusion coefficient we have
\begin{cor}\label{cor:uniconti}
Let $D$ be the matrix defined in Theorem \ref{thm:coe}. Then, for each $1 \le i \le d$,
\begin{displaymath}
\inf_{\mathfrak{f} \in \Cy_0} \sup_{-1 \le \rho \le 1} \ll W_{0,e_i} + \sum_{j=1}^d D_{i,j}(\rho)[\eta(e_j) - \eta(0)] -L\mathfrak{f}(\eta) \gg_{\rho} =0 .
\end{displaymath}
\end{cor}
This result together with (\ref{eq:variance}), the definition of $\tilde{V}_i^{\mathfrak{f},l}$ and Theorem \ref{thm:converge} concludes the proof of Theorem \ref{thm:keythm}.

We conclude this section proving that the diffusion coefficient $D$ is a diagonal matrix and it has the same diagonal component, therefore $D(\rho)=d(\rho)I$.

\begin{thm}\label{thm:diagonal}
There exists a continuous function $d(\rho)$ on $[-1,1]$ such that $D(\rho)=d(\rho)I$ and 
\begin{displaymath}
\frac{\chi(\rho)}{4\langle \Psi_{0,e_1}^2 \rangle_{\rho}}\le d(\rho) \le \frac{\langle (\nabla_{0,e_1}\eta(0))^2 \rangle_{\rho}}{\chi(\rho)}. 
\end{displaymath}
\end{thm}
\begin{proof}
Because of the symmetry of the dynamics, it is obvious that $D$ has the same diagonal component. It remains to show that $D$ is a diagonal matrix.

According to \cite{S}, the diffusion coefficient matrix defined by the variational formula (\ref{eq:inf}) coincides with the diffusion coefficient matrix defined by the Green-Kubo formula based on the current-current correlation function:
\begin{displaymath}
a^{*}D(\rho)a:=\frac{1}{\chi(\rho)} \Big\{\sum_{i=1}^d a_i^2\langle (\nabla_{0,e_i} \eta(0))^2 \rangle_{\rho} -\frac{1}{2}\int^{\infty}_0 \sum_{x \in \Z^d}E_{\nu_\rho}[W_ae^{Lt}\tau_xW_a]dt \Big\}
\end{displaymath}
where $W_a:=\sum_i a_i W_{0,e_i}$. Therefore, we have only to prove that 
\begin{displaymath}
\int^{\infty}_0 \sum_{x \in \Z^d}E_{\nu_\rho}[W_{0,e_i}e^{Lt}\tau_xW_{0,e_j}]dt=0
\end{displaymath}
for all $i \neq j$. In \cite{KV}, Kipnis and Varadhan proved some equivalent relation about the central limit theorem variance. We can use one of them. It holds that
\begin{displaymath}
\int^{\infty}_0 \sum_{x \in \Z^d}E_{\nu_\rho}[W_{0,e_i}e^{Lt}\tau_xW_{0,e_j}]dt=\lim_{\la \to 0}\sum_xE_{\nu_\rho}[W_{0,e_i}\tau_xg_{\la}^j]
\end{displaymath}
where $g_{\la}^j$ is a solution of the resolvent equation $\la g_{\la}^j - L g_{\la}^j=W_{0,e_j}$.

Denote by $\theta_i$ the reflection operator with respect to $\frac{1}{2}e_i$ along the $e_i$ direction, namely for $x \in \Z^d$, $\theta_i x = (x_1, x_2, ...,  x_{i-1}, -x_i+1, x_{i+1}, ... ,x_d)$. We may extend $\theta_i$ to configurations in $\chi^d$ and to functions on $\chi^d$ naturally:
\begin{displaymath}
(\theta_i \eta)(x):=\eta(\theta_i x) \quad (\theta_i f)(\eta):=f(\theta_i \eta).
\end{displaymath}
Then, for $i \neq j$, 
\begin{displaymath}
\la \ \theta_i \tau_x g_{\la}^j - L \ \theta_i \tau_x g_{\la}^j= \theta_i \tau_x W_{0,e_j} = \tau_{\theta_i x - 2e_i} W_{0,e_j}.
\end{displaymath}
Therefore, since $\nu_\rho$ is translation invariant and a product measure,
\begin{displaymath}
E_{\nu_\rho}[W_{0,e_i}\tau_xg_{\la}^j]=E_{\nu_\rho}[\theta_iW_{0,e_i} \theta_i\tau_xg_{\la}^j]= E_{\nu_\rho}[-W_{0,e_i} \tau_{\theta_i x - 2e_i}g_{\la}^j].
\end{displaymath}
Since the map $x \to (\theta_i x - 2e_i)$ is a bijection,
\begin{displaymath}
\sum_xE_{\nu_\rho}[W_{0,e_i}\tau_xg_{\la}^j]=\sum_xE_{\nu_\rho}[-W_{0,e_i}\tau_xg_{\la}^j].
\end{displaymath}
Thus, $\sum_xE_{\nu_\rho}[W_{0,e_i}\tau_xg_{\la}^j]=0$ for all $\la$.
\end{proof}

\begin{rem}
If we assume the gradient condition $C_++C_- - C_A-2C_E=0$, then $W_{0,e_i}=h(\eta(0))-h(\eta(e_i))$ with $h(-1)=C_+, h(0)=0$ and $h(-1)=-C_-$. In this case, $\ll W_{0,e_i},Lg \gg_{\rho}=0$ holds. Therefore $d(\rho)= \langle (\nabla_{0,e_i} \eta(0))^2 \rangle_{\rho}=-\frac{\Phi^{\prime}(\rho)}{2}(C_+ - C_-)+\frac{1}{2}(C_+ + C_-)$.
\end{rem}

\subsection{Spectral Gap}
In this section, we prove the spectral gap for the two-species exclusion process on finite $d$-dimensional cubes. For a positive integer $N$, we denote by $\Om_N$ the box $\{1,...,N\}^d$ and by $\mathcal{Y}_N$ the space of configurations $\{-1,0,1\}^{\Om_N}$. Let $L_{\Om_N}$ be the generator of the two-species exclusion process on $\Om_N$ with free boundary conditions:
\begin{displaymath}
L_{\Om_N}f(\eta)=\sum_{x,y \in \Om_N, |x-y|=1}L_{xy}f(\eta)
\end{displaymath}
where $L_{xy}$ was defined in (\ref{eq:generator}).

For $-|\Om_N| \le K \le |\Om_N|$, we denote by $\mathcal{Y}_{N,K}$ the hyperplane $\{\eta; \sum_{x \in \Om_N}\eta(x)=K\}$ and by $\mu_{N,K}$ the product measure $\nu_{\rho}$ on $\mathcal{Y}_N$ conditioned on the hyperplane $\mathcal{Y}_{N,K}$:
\begin{displaymath}
\mu_{N,K}(\cdot)=\nu_{\rho}( \ \cdot \ | \sum_{x \in \Om_N}\eta(x)=K).
\end{displaymath}
As in the previous sections, expected values with respect to the measure $\mu_{N,K}$ are denoted by $\langle \cdot \rangle_{N,K}$:
\begin{displaymath}
\langle f \rangle_{N,K}:=\int_{\mathcal{Y}_{N,K}}f(\eta)\mu_{N,K}(d\eta).
\end{displaymath}

In the main theorem of this section we prove that the generator $L_{\Om_N}$ in $L^2(\mu_{N,K})$ has a spectral gap of order at least $N^{-2}$.
\begin{thm}\label{thm:spectralgap}
There exists a positive constant $C$, which only depends on the constants $C_+, C_-, C_A, C_C$ and $C_E$, such that for every positive integer $N$, every integer $-|\Om_N| \le K \le |\Om_N|$ and every function $f$ in $L^2(\mu_{N,K})$ satisfying $ \langle f \rangle _{N,K}=0$,
\begin{displaymath}
 \langle f^2 \rangle _{N,K} \le CN^2 \langle -L_{\Om_N}f,f \rangle _{N,K}.
\end{displaymath}
\end{thm}

We start with showing that $\langle -L_{\Om_N}f,f \rangle _{N,K}$ is bounded below by $C\langle -\tilde{L}_{\Om_N}f,f \rangle _{N,K}$ with some constant $C$ where $\tilde{L}_{\Om_N}$ acting on functions as
\begin{displaymath}
\tilde{L}_{\Om_N}f(\eta)=\sum_{x,y \in \Om_N, |x-y|=1}\tilde{L}_{xy}f(\eta)
\end{displaymath}
and
\begin{align*}
\tilde{L}_{xy}&f(\eta)=[\Psi^{x,y}_{1,0}(\eta) + \Psi^{x,y}_{0,-1}(\eta) + \Psi^{x,y}_{-1,1}(\eta)](f(\eta^{x,y})-f(\eta)) \\
&+ \Psi^{x,y}_{1,-1}(\eta)(f(\eta^{x=0,y=0})-f(\eta)) + \b \Psi^{x,y}_{0,0}(\eta)(f(\eta^{x=-1,y=1})-f(\eta)).
\end{align*}
Notice that $\tilde{L}_{\Om_N}$ is the generator of two-species exclusion process with $C_+=C_-=C_A=C_E=1$ and $C_C=\b$. The probability measures $\mu_{N,K}$ are also reversible for the Markov process with generator $\tilde{L}_{\Om_N}$.
\begin{lem}
If we assume that $C_+,C_-,C_A,C_C$ are all positive constants and $C_E$ is a nonnegative constant, there exists a positive constant $C$ such that for every positive integer $N$, every integer $-|\Om_N| \le K \le |\Om_N|$, every function $f$ in $L^2(\mu_{N,K})$ and every directed bond $b=(x,y)$ we have
\begin{displaymath}
\langle (-\tilde{L}_{xy}-\tilde{L}_{yx})f,f \rangle _{N,K} \quad \le \quad C  \langle (-L_{xy}-L_{yx})f,f \rangle _{N,K}
\end{displaymath}
\end{lem}
\begin{proof}
It is enough to prove the lemma assuming $C_E=0$. Especially, we only have to bound the term $\langle \Psi^{x,y}_{-1,1}(\eta)(f(\eta^{x,y})-f(\eta))^2 \rangle _{N,K}$ by the term $C \langle (-L_{xy}-L_{yx})f,f \rangle _{N,K}$ with some constant $C$. By the Cauchy-Shwartz inequality, we have
\begin{align*}
& \langle \Psi^{x,y}_{-1,1}(\eta)(f(\eta^{x,y})-f(\eta))^2 \rangle _{N,K} =  \langle \Psi^{x,y}_{1,-1}(\eta)(f(\eta^{x,y})-f(\eta))^2 \rangle _{N,K} \\
&\le \ 2  \langle \Psi^{x,y}_{1,-1}(\eta)[(f(\eta^{x,y})-f(\eta^{x=0,y=0}))^2 + (f(\eta^{x=0,y=0})-f(\eta))^2] \rangle _{N,K} \\
\end{align*}
and the last expression is written as
\begin{displaymath}
2\b  \langle \Psi^{x,y}_{0,0}(\eta)(f(\eta^{x=-1,y=1})-f(\eta))^2 \rangle _{N,K} +2  \langle \Psi^{x,y}_{1,-1}(\eta)(f(\eta^{x=0,y=0})-f(\eta))^2 \rangle _{N,K}
\end{displaymath}
by change of variables. Therefore, we can obtain the desirable estimate with the constant $C:=\min\{C_+,C_-, \frac{C_A}{3}\}$. 
\end{proof}

Now, to conclude the proof of Theorem \ref{thm:spectralgap}, we have only to prove the theorem as follows:
\begin{thm}\label{thm:spectraltilde}
There exists a positive constant $C$ such that for every positive integer $N$, every $-|\Om_N| \le K \le |\Om_N|$ and every function $f$ in $L^2(\mu_{N,K})$ satisfying $ \langle f \rangle _{N,K}=0$,
\begin{displaymath}
\langle f^2 \rangle _{N,K} \le CN^2 \langle -\tilde{L}_{\Om_N}f,f \rangle _{N,K}.
\end{displaymath}
\end{thm}
The proof of this theorem relies on the study of the spectral gap of the two-species exclusion process of mean field type. This is the Markov process on $\mathcal{Y}_N$, whose generator $\tilde{L}_{\Om_N}^m$ acting on functions $f$ as
\begin{displaymath}
\tilde{L}_{\Om_N}^m f(\eta)=\frac{1}{|\Om_N|}\sum_{x,y \in \Om_N}\tilde{L}_{xy}f(\eta).
\end{displaymath}
Notice that the probability measures $\mu_{N,K}$ are also reversible for the Markov process with generator $\tilde{L}_{\Om_N}^m$. This generator has a spectral gap in $L^2(\mu_{N,K})$ of order at least 1 as stated in the next theorem.

\begin{thm}\label{thm:spectraluni}
There exists a finite constant $C$ such that for every positive integer $N$, every integer $-|\Om_N| \le K \le |\Om_N|$ and every function $f$ in $L^2(\mu_{N,K})$ satisfying $ \langle f \rangle _{N,K}=0$,
\begin{displaymath}
 \langle f^2 \rangle _{N,K} \le C \langle -\tilde{L}_{\Om_N}^m f,f \rangle _{N,K}.
\end{displaymath}
\end{thm}
Before proving Theorem \ref{thm:spectraluni}, we show that Theorem \ref{thm:spectraltilde} is an easy corollary of this result.
\begin{proof}[Proof of Theorem \ref{thm:spectraltilde}]
For each pair $\{x,y\} \in \Om_N \times \Om_N$, we determine a path inside $\Om_N$ which connects $x=(x_1,...,x_d)$ and $y=(y_1,...,y_d)$ as follows: First we connect $x$ and $(y_1,x_2,...,x_d)$ only by changing the first coordinate one by one. Then, $(y_1,x_2,x_3...,x_d)$ and $(y_1,y_2,x_3...,x_d)$ are connected by changing the second coordinate and this procedure is continued. We denote the sequence of bonds appearing in this path by $b_1=(z_1,w_1), b_2,...,b_M=(z_M,w_M)$ and the set of these bonds by $\mathbb{B}(x,y)$. For a configuration $\eta$ satisfying $r_{x,y}(\eta) \neq 0$, let define a sequence of configurations $(\xi_i)_{0 \le i \le 2M-1}$ such that $\xi_0=\eta$, $\xi_{2M-1}=\eta^{(x,y)}$ as follows : $\xi_0:=\eta$, $\xi_j:=(\xi_{j-1})^{z_j,w_j}$ for $j \le M-1$, $\xi_M=(\xi_{M-1})^{(z_M,w_M)}$ and $\xi_j:=(\xi_{j-1})^{z_{2M-j},w_{2M-j}}$ for $M+1 \le j \le 2M-1$. Then, we have
\begin{align*}
\langle \Psi^{x,y}_{1,0}(\eta)  (f(\eta^{(x,y)}) & -f(\eta))^2 \rangle _{N,K}  = \langle \Psi^{x,y}_{1,0}(\eta)[\sum_{0 \le i \le 2M-1}(f(\xi_{i+1})-f(\xi_i))]^2 \rangle _{N,K}\\
	& \le (2M-1) \sum_{0 \le i \le 2M-1}  \langle r_{z_{i+1},w_{i+1}}(\eta)(f(\eta^{z_{i+1},w_{i+1}})-f(\eta))^2 \rangle _{N,K}\\
	& \le 4dN  \sum_{b \in \mathbb{B}(x,y)}  \langle (-\tilde{L}_b-\tilde{L}_{b^{\prime}})f,f \rangle _{N,K}.\\
\end{align*}
Similarly, we have
\begin{displaymath}
\langle (-\tilde{L}_{xy}-\tilde{L}_{yx})f,f \rangle_{N,K} \le CN \sum_{b \in \mathbb{B}(x,y)} \langle (-\tilde{L}_b-\tilde{L}_{b^{\prime}})f,f \rangle _{N,K}\\
\end{displaymath}
for some positive constant $C$.

Applying Theorem \ref{thm:spectraluni}, for all functions $f$ in $L^2(\mu_{N,K})$ satisfying $ \langle f \rangle _{N,K}=0$ we obtain that   
\begin{align*}
 \langle f^2 \rangle _{N,K} & \le C \frac{1}{|\Om_N|}\sum_{x,y \in \Om_N} \langle -\tilde{L}_{xy}f,f \rangle _{N,K}  \le C \frac{N}{|\Om_N|}\sum_{x,y \in \Om_N}\sum_{b \in \mathbb{B}(x,y)} \langle (-\tilde{L}_b-\tilde{L}_{b^{\prime}})f,f \rangle _{N,K} \\
	& \le C \frac{N}{|\Om_N|}\sum_{b \in (\Om_N)^*} \langle -\tilde{L}_bf,f \rangle _{N,K} \times \#\{(x,y) \in \Om_N \times \Om_N; b \ \text{or} \ b^{\prime} \in \mathcal{B}(x,y)\} \\
	& \le CN^2 \sum_{b \in (\Om_N)^*} \langle -\tilde{L}_bf,f \rangle _{\La_N,K} 
\end{align*}
where a constant $C$ changes each line.
\end{proof}

\begin{proof}[Proof of Theorem \ref{thm:spectraluni}]
There is a duality between $+$particles and $-$particles, i.e., $+$particles evolve with the same dynamics as $-$particles do under the generator $\tilde{L}_{\Om_N}^u$. Therefore, we assume that $0 \le K \le |\Om_N|$.

Let $X(\eta)$ denote the number of sites occupied by $-$particles in the configuration $\eta$:
\begin{displaymath}
X(\eta):=\sum_{x \in \Om_N}1_{\{\eta(x)=-1\}}.
\end{displaymath}
We first project $f$ on the $\sigma$-field generated by $X$ and on its orthogonal:
\begin{equation}
 \langle f^2 \rangle _{N,K}  =  \langle  (f-E[f|X])^2  \rangle _{N,K} +  \langle (E[f|X])^2 \rangle _{N,K}. \label{eq:twoterm}
\end{equation}
We consider the two terms separately. 
Let us define $L^1_{xy}$ for each ordered pair $(x,y)$ by
\begin{displaymath}
(L^1_{xy} f)(\eta) =[\Psi^{x,y}_{1,0}(\eta) + \Psi^{x,y}_{0,-1}(\eta) + \Psi^{x,y}_{-1,1}(\eta)](f(\eta^{x,y})-f(\eta)) 
\end{displaymath}
and define $L^1_{\Om_N}$ by
\begin{displaymath}
L^1_{\Om_N}f(\eta)=\frac{1}{|\Om_N|}\sum_{x,y \in \Om_N}L^1_{xy}f(\eta).
\end{displaymath}
To bound the first term in (\ref{eq:twoterm}) by the Dirichlet form $\langle -L^1_{\Om_N}f, f \rangle$, we use a general result concerning the spectral gap for multispecies exclusion processes. 

We introduce some notation. For positive integers $r$, $N$ and nonnegative integers $K_1,...K_r$ such that $\sum_{i=1}^r K_i\le N$, define $\Si_{N,K_1,...K_r}^r$ as the hyperplane of all configurations of $\Si_N^r:=\{0,1,2...r\}^N$ with $K_i$ sites occupied by the $i$-particles:
\begin{displaymath}
\Si_{N,K_1,...K_r}^r=\{\eta \in \Si_N^r; \sum_{j=1}^N \ 1_{\{\eta(j)=i\}}=K_i \ 1 \le i \le r\}
\end{displaymath}
and $m_{N,K_1,...K_r}$ as the uniform probability measure on $\Si_{N,K_1,...K_r}^r$. As before we denote by $\langle \cdot \rangle_{N,K_1,...K_r}$, the expectation with respect to the measure $m_{N,K_1,...K_r}$. Consider the process that exchanges the value of configurations between any two sites at a fixed rate. Its generator $L^r_N$ is given by 
\begin{displaymath}
 L_N^rf(\eta) = \frac{1}{N}\sum_{1 \le j,k \le N}(f(\eta^{j,k})-f(\eta))
\end{displaymath}
where 
\begin{displaymath}
 \eta^{j,k}(z) = \begin{cases}
	\eta(z)  & \text{if $z \neq x,y$} \\
	\eta(k) & \text{if $z=j$} \\
	\eta(j)  & \text{if $z=k$.} 
\end{cases} 
\end{displaymath}
A simple computation shows that the uniform measures $m_{N,K_1,...K_r}$ are reversible for this process. We prove that the spectral gap of the generator $L^r_N$ is of order $O(1)$.
\begin{prop}\label{prop:multispectral}
There exists a positive constant $C=C(r)$ such that for every positive integer $N$, every set of nonnegative integers $K_1,...K_r$ such that $\sum_{i=1}^r K_i \le N$ and every function $f$ in $L^2(m_{N,K_1,...K_r})$ satisfying $ \langle f \rangle_{N,K_1,...K_r}=0$,
\begin{equation}
 \langle f^2 \rangle_{N,K_1,...K_r} \le C \langle -L^r_Nf,f \rangle_{N,K_1,...K_r}. \label{eq:multispectral}
\end{equation}
\end{prop}
The proof of this proposition is postponed to the last part of this section. To apply the estimate in (\ref{eq:multispectral}), we rewrite the first term of in the right hand side of (\ref{eq:twoterm}) as
\begin{displaymath}
 \langle  (f-E[f|X])^2  \rangle _{N,K} = \sum_{l=0}^{\frac{|\Om_N| - K}{2}} \mu_{N,K}(\{X=l\})  \langle f_l^2 \rangle _{\La_N,K,l}
\end{displaymath}
where $ \langle  \cdot  \rangle _{\La_N,K,l}$ stands for the expectation with respect to the uniform measure on the set of configurations $\eta \in \{-1, 0, 1\}^{\Om_N}$ satisfying $\sum_{x \in \Om_N}1_{\{\eta(x)=1\}}=K+l$ and $\sum_{x \in \Om_N}1_{\{\eta(x)=-1\}}=l$, and $f_l$ stands for the function on this set defined by $f_l(\eta)=f(\eta)-E[f|X](\eta)$. 
By Proposition \ref{prop:multispectral}, we have that
\begin{eqnarray*}
 \langle  (f-E[f|X])^2  \rangle _{N,K} &\le& C\sum_{l=0}^{\frac{|\Om_N| - K}{2}} \mu_{N,K}(\{X=l\}) \langle -L^1_{\Om_N}f_l,f_l \rangle _{\La_N,K,l} \\
	&=& C \langle -L^1_{\Om_N}f,f \rangle _{N,K} \\
	&\le&  C\langle -\tilde{L}_{\Om_N}f,f \rangle _{N,K}, \\
\end{eqnarray*}
notice that $f_l$ can be replaced by $f$ to have the second line.

Next, let us consider the second term of (\ref{eq:twoterm}). Let $\eta_t$ be the Markov process with the generator $\tilde{L}_{\Om_N}^m$. Since the original geometry of the process evolving according to the generator $L_{\Om_N}$ is lost, $X(\eta_t)$ is a Markov process and the state space of this Markov process is $\chi:=\{0,1,...,\frac{|\Om_N|-K}{2}\}$. A simple computation shows that its generator is given by 
\begin{displaymath}
\mathcal{L}_{N,K}f(l)=r(l,l-1)(f(l-1)-f(l))+r(l,l+1)(f(l+1)-f(l)) \quad \forall l \in \chi
\end{displaymath}
where $r(l,l-1)=\frac{l(K+l)}{|\Om_N|}$ and $r(l,l+1)=\frac{(|\Om_N|-K-2l)(|\Om_N|-K-2l-1)\b}{|\Om_N|}$. For fixed $N$ and $K$, denote by $\tilde{m}_{N,K}$ the probability measure $\mu_{N,K}X^{-1}$ on $\chi$:
\begin{displaymath}
\tilde{m}_{N,K}(l):=\mu_{N,K}(\{X=l\}).
\end{displaymath}
For $X$-measurable function $f$, define a function $\tilde{f}:\chi \to \R$ by $\tilde{f}(l):=f(\eta)$ for some $\eta$ such that $X(\eta)=l$. A simple computation shows that $\tilde{L}^m_Nf=\mathcal{L}_{N,K}\tilde{f}$ and $\langle f \rangle _{N,K} = \langle \tilde{f} \rangle _{\tilde{m}_{N,K}}$. Therefore, $\langle -\tilde{L}^u_Nf, f \rangle _{N,K}= \langle -\mathcal{L}_{N,K}\tilde{f},\tilde{f} \rangle _{\tilde{m}_{N,K}}$ and $\langle f^2 \rangle _{N,K} = \langle \tilde{f}^2 \rangle _{\tilde{m}_{N,K}}$ hold. To conclude the proof of Theorem \ref{thm:spectraluni}, we have only to prove Lemma \ref{lem:asym} below.
\end{proof}

\begin{lem}\label{lem:asym}
There exists a constant $C_{\b}$ such that for any integer $N$ and $K$ satisfying $0 \le K \le |\Om_N|$
\begin{displaymath}
\langle f^2 \rangle _{\tilde{m}_{N,K}} \le C_{\b} \langle -\mathcal{L}_{N,K}f,f \rangle _{\tilde{m}_{N,K}}
\end{displaymath}
for all functions $f:\chi \to \R$ satisfying $ \langle f \rangle _{\tilde{m}_{N,K}}=0$ where $\tilde{m}_{N,K}$, $\mathcal{L}_{N,K}$ and $\chi$ were defined above. 
\end{lem}

The proof of this lemma is based on a general result concerning the spectral gap for strongly asymmetric reversible Markov processes presented below, see \cite{KLO}.
\begin{prop}\label{prop:asym}
Let $(X_t)$ be a Markov process with generator denoted by $L$ on a countable state space $E$ reversible with respect to a probability measure $m$, where $L$ acting on functions as $Lf(x)=\sum_{y \in E}r(x,y)(f(y)-f(x))$ for $x \in E$. Suppose that there exists a point $e_0 \in E$, a positive constant $C_0$ and a ramification $\{\gamma(e_0,x);x \in E\}$ satisfying the following assumption (H): $|\gamma(e_0,x)| \le C_0 \ [r(x,p(x))-\sum_{y \in s(x)}r(x,y)]$. Then, for every $f$ in $L^2(m)$, we have
\begin{displaymath}
 \langle (f- \langle f \rangle _m)^2 \rangle _m \le 2C_0  \langle -Lf,f \rangle _m 
\end{displaymath}
\end{prop}

\begin{proof}[Proof of Lemma \ref{lem:asym}]
For $K=2|\Om_N|-1$ and $2|\Om_N|$ the process has only one possible state. Therefore, we may assume that $0 \le K \le 2|\Om_N|-2$.

Denote by $\xi_t$ the Markov process on $\chi$ with generator $\mathcal{L}_{N,K}$. The proof consists in finding a state $e_0$ and a ramification $\{\gamma(e_0,x),x \in \chi_n\}$ satisfying the assumption $(H)$ required in Proposition \ref{prop:asym} for some strictly positive constant $C_0$.

The natural candidate as root of the ramification is the point where the drift  of the particle is 0. Define $D(x)$ as the mean drift of the particle at $x$: $D(x)=r(x,x+1)-r(x,x-1)$. Since $0 \le K \le 2|\Om_N|-2$, a simple computation shows that $D(0) \ge 0$ and $D([\frac{|\Om_N|-K}{2}]) \le 0$. Let $\bar{e_0}$ be the unique root of $D(x)$ in $[0,[\frac{|\Om_N|-K}{2}]]$ and define $e_0$ as the nearest integer of $\bar{e_0}$.

Once $e_0$ is defined, there is only one possible ramification of the state space. For $x \ge e_0$ we have to define the path from $e_0$ to $x$ as $\gamma(e_0,x)=(e_0, e_0+1,...,x)$. In the same way for $x \le e_0$ the path from $e_0$ to $x$ has to be $\gamma(e_0,x)=(e_0, e_0-1,...,x)$. With this ramification, for a fixed $x \ge e_0$ the parent of $x$ is $x-1$ and there is only one child $x+1$. In the same way for a fixed $x \le e_0$ the parent of $x$ is $x+1$ and there is only one child $x-1$. Therefore, in order to prove that this ramification satisfies assumption (H) of the Proposition \ref{prop:asym}, we have to show that $r(x,x+1)-r(x,x-1) \ge  C_0 (e_0-x)$ for $x < e_0$ and $r(x,x-1)-r(x,x+1) \ge  C_0 (x-e_0)$ for $x > e_0$.

A simple computation shows that $D^{\prime}(x) \le -C_{\b}^N$ where $C_{\b}^N=\frac{|\Om_N|-2\b}{|\Om_N|}$ for $\b \ge \frac{1}{4}$ and $C_{\b}^N=2\b\frac{2|\Om_N|-1}{|\Om_N|}$ for $\b \le \frac{1}{4}$. Therefore, for $x > e_0$ and large N,
\begin{align*}
r(x,x-1)-r(x,x+1) = -D(x) &\ge C_{\b}^N(x-\bar{e_0}) \\
			&\ge C_{\b}^N(x-e_0)-C_{\b}^N|e_0-\bar{e_0}| \\
			&\ge \frac{C_{\b}^N}{2}(x-e_0).
\end{align*}
The last inequality follows from the definition of $e_0$ since $|e_0-\bar{e_0}| \le \frac{1}{2}$ and from the fact that $x \ge e_0 +1$. In the same way we can prove that for $x < e_0$,
\begin{displaymath}
r(x,x+1)-r(x,x-1) \ge \frac{C_{\b}^N}{2}(e_0-x).
\end{displaymath}
Since there exists some finite constant $C_{\b}$ such that $\frac{2}{C_{\b}^N} \le C_{\b}$ for all $N$, we conclude the proof of the Lemma \ref{lem:asym}.
\end{proof}

\begin{proof}[Proof of Proposition \ref{prop:multispectral}]
The idea of the proof consists in using the means of the mathematical induction with respect to $r$.
Assume $r=1$. Then, the process is the uniform symmetric simple exclusion process for which Quastel proved a spectral gap in [8].

Next, consider the general positive integer $r$. First of all, we suppose without loss of generality that $K_0 \le K_i$ for $0 \le i \le r$ where $K_0:=N-\sum_{i=1}^rK_i$. This implies that $K_0 \le \frac{N}{r+1}$.
For $0 \le i \le r$, define $\pi_0:\Si^r_N \to \{0,1\}^N$ as a function on the configuration space which do not distinguish sites occupied by some particle:
\begin{displaymath}
\pi_0(\eta)=\xi \in \{0,1\}^N \ \text{where} \ \xi(j)=
\left\{
\begin{aligned}
&1&  \eta(j)=0 \\
&0&  \text{otherwise.}
\end{aligned}
\right.
\end{displaymath}
For a function $f$ in $L^2(m_{N,K_1,...K_r})$ satisfying $ \langle f \rangle_{N,K_1,...K_r}=0$, define $f_0$ as the conditional expectation of $f$ with respect to the $\si$-field generated by $\pi_0$. Then, similar to the proof of Theorem \ref{thm:spectraluni}, we project $f$ on this $\si$-field and on its orthogonal:
\begin{equation}
 \langle f^2 \rangle _{m_{N,K_1,...K_r}}  =  \langle  (f-f_0)^2  \rangle _{m_{N,K_1,...K_r}} +  \langle f_0^2 \rangle _{m_{N,K_1,...K_r}}. \label{eq:twoterm2}
\end{equation}
We first consider the second term. The arguments are similar to the ones used in the second step of the proof of Theorem \ref{thm:spectraluni}. 

We may think $f_0$ as a function defined on $\{0,1\}^N$. The generator $L^r_N$ acting on $\{0,1\}^N$ is the generator of the usual uniform symmetric exclusion process for which Quastel [8] proved a spectral gap. Therefore, we have 
\begin{displaymath}
\langle f_0^2 \rangle _{m_{N,K_1,...K_r}} \le \langle L^r_N f_0, f_0 \rangle _{m_{N,K_1,...K_r}} \le \langle L^r_N f, f \rangle _{m_{N,K_1,...K_r}}. 
\end{displaymath}
We now turn to the first term of (\ref{eq:twoterm2}). For a subset $B \subset \La_N:=\{1,...,N\}$ such that $\#B=K_0$, define $f_{0,B}$ as a function on $\Si_{\La_N-B,K_1,K_2,...,K_{r-1}}^{r-1}$ as follows:
\begin{displaymath}
f_{0,B}(\xi)=f(\xi^B)-f_0(\xi^B) \quad \xi \in \Si_{\La_N-B,K_1,K_2,...,K_{r-1}}^{r-1}.
\end{displaymath}
In this formula, $\xi^B$ stands for the configuration $\eta \in \Si_{N,K_1,K_2,...,K_r}^r$ such that 
\begin{displaymath}
\eta(i)= \begin{cases}
0 & \text{if} \quad i \in B \\
j+1 & \text{if} \quad \xi(i)=j.
\end{cases}
\end{displaymath}
With this notation, we rewrite the first term of (\ref{eq:twoterm2}) as follows:
\begin{align*}
\langle (f-f_0)^2 & \rangle _{m_{N,K_1,...K_r}} \\
& = \sum_{B \subset \La_N}m_{N,K_1,...K_r}(\{\eta; \eta(i)=0 \ \text{for all} \ i \in B\}) \langle  f_{0,B}^2  \rangle _{m_{N-K_0,K_1,...K_{r-1}}}.
\end{align*}
The generator $L^r_N$ acting on $\Si_{\La_N-B,K_1,K_2,...,K_{r-1}}^{r-1}$ is the operator $\frac{N-K_0}{N} L^{r-1}_{N-K_0}$. By the induction assumption, there exists some constant $C_{r-1}$ such that
\begin{align*}
&\langle f_{0,B} \rangle _{m_{N-K_0,K_1,...K_{r-1}}} \le C_{r-1} \langle L^{r-1}_{N-K_0} f_{0,B}, f_{0,B} \rangle _{m_{N-K_0,K_1,...K_{r-1}}} \\
& \le  C_{r-1} \frac{N}{N-K_0} \langle L^r_N f, f \rangle _{m_{N,K_1,...K_r}}.
\end{align*}
By the assumption $K_0 \le \frac{N}{r+1}$, therefore we have $ \frac{N}{N-K_0} \le \frac{r+1}{r}$ and
\begin{displaymath}
\langle f_{0,B} \rangle _{m_{N-K_0,K_1,...K_{r-1}}} \le C_r \langle L^r_N f, f \rangle _{m_{N,K_1,...K_r}} 
\end{displaymath}
for some constant $C_r$.
\end{proof}
\subsection{Closed Forms}	
In this section, we introduce the notion of closed forms associated to our model and those of the generalized exclusion process. We have one-to-one correspondence between them, so algebraic characterization of the closed forms can be reduced to that for the generalized exclusion process.

Let $\mathcal{H}_{x, x+e_i}$ be a subspace of $\chi=\{-1,0,1\}^{\Z^d}$ such that
\begin{displaymath}
\mathcal{H}_{x, x+e_i}:=\{\eta \ ;\Psi^{x,x+e_i}_{1,0}(\eta)+\Psi^{x,x+e_i}_{0,-1}(\eta)+\Psi^{x,x+e_i}_{-1,1}(\eta)+\Psi^{x,x+e_i}_{1,-1}(\eta)+\Psi^{x,x+e_i}_{0,0}(\eta)=1 \}.
\end{displaymath}

For two configurations $\eta$ and $\xi \in \chi$, define $D(\eta,\xi)$ as follows: $D(\eta,\xi)=1$ if there exists a unique point $x \in \Z^d$ and a unique direction $i$ such that $\eta(z)=\xi(z)$ for all $z \neq x, x+e_i$ and $\eta(x)+\eta(x+e_i)=\xi(x)+\xi(x+e_i)$ and $\eta \neq \xi$, and $D(\eta,\xi)=0$ otherwise.  A path $\Gamma(\eta,\xi)=(\eta=\eta^0, \eta^1,...,\eta^{m-1},\eta^m=\xi)$ from $\eta$ to $\xi$ is a sequence of configurations $\eta^j$ such that every two consecutive configurations satisfies $D(\eta^j,\eta^{j+1})=1$. 

Consider a family of continuous functions $\mathfrak{u}=(\mathfrak{u}_x^i)_{1 \le i \le d, x \in \Z^d}$ where $\mathfrak{u}_x^i: \mathcal{H}_{x,x+e_i} \to \R$. For an ordered pair $(\eta, \xi)$ satisfying $D(\eta,\xi)=1$, define a path integral $I^1_{(\eta,\xi)}$ of $\mathfrak{u}$ by 
\begin{align*}
I^1_{(\eta,\xi)}(\mathfrak{u}) &:= 
	\frac{1}{\sqrt{C_+}} \big( \mathfrak{u}_x^i(\eta) \Psi^{x,x+e_i}_{1,0}(\eta)- \mathfrak{u}_x^i(\xi) \Psi^{x,x+e_i}_{1,0}(\xi) \big) \\
	&+\frac{1}{\sqrt{C_-}} \big( \mathfrak{u}_x^i(\eta) \Psi^{x,x+e_i}_{0,-1}(\eta)- \mathfrak{u}_x^i(\xi) \Psi^{x,x+e_i}_{-1,0}(\xi) \big) \\
	&+ \frac{1}{\sqrt{C_E}} \big( \mathfrak{u}_x^i(\eta) \Psi^{x,x+e_i}_{-1,1}(\eta) \Psi^{x,x+e_i}_{1,-1}(\xi) - \mathfrak{u}_x^i(\xi) \Psi^{x,x+e_i}_{1,-1}(\eta) \Psi^{x,x+e_i}_{-1,1}(\xi) \big) \\
	&+ \frac{1}{\sqrt{C_A}} \big( \mathfrak{u}_x^i(\eta) \Psi^{x,x+e_i}_{1,-1}(\eta) \Psi^{x,x+e_i}_{0,0}(\xi) - \mathfrak{u}_x^i(\xi) \Psi^{x,x+e_i}_{0,0}(\eta) \Psi^{x,x+e_i}_{1,-1}(\xi) \big) \\
	&+ \frac{1}{\sqrt{C_C}} \big( \mathfrak{u}_x^i(\eta) \Psi^{x,x+e_i}_{0,0}(\eta) \Psi^{x,x+e_i}_{-1,1}(\xi) - \mathfrak{u}_x^i(\xi) \Psi^{x,x+e_i}_{-1,1}(\eta) \Psi^{x,x+e_i}_{0,0}(\xi) \big) \\
\end{align*}
A path integral can be naturally extended to paths of any length as
\begin{displaymath}
I^1_{\Gamma(\eta,\xi)}(\mathfrak{u}):=\sum_{j=0}^{m-1} I^1_{(\eta^j,\eta^{j+1})}(\mathfrak{u}).
\end{displaymath}

A family of continuous functions $(\mathfrak{u}_x^i)_{1 \le i \le d, x \in \Z^d}$ is called an $I^1$-closed form if for all closed path $\Gamma(\eta,\xi)$, $I^1_{\Gamma(\eta,\xi)}(\mathfrak{u})=0$ where a path $\Gamma(\eta,\xi)$ is called closed if $\eta=\xi$.

Next, let us recall the path integral and the closed form of the generalized exclusion processes. Let $\widetilde{\mathcal{H}_{x, x+e_i}}$ be a subspace of $\chi=\{-1,0,1\}^{\Z^d}$ such that
\begin{displaymath}
\widetilde{\mathcal{H}_{x, x+e_i}}:=\{\eta \ ;\Psi^{x,x+e_i}_{1,0}(\eta)+\Psi^{x,x+e_i}_{0,-1}(\eta)+\Psi^{x,x+e_i}_{1,-1}(\eta)+\Psi^{x,x+e_i}_{0,0}(\eta)=1 \}.
\end{displaymath}

For two configurations $\eta$ and $\xi \in \chi$, define $\tilde{D}(\eta,\xi)$ as follows: $\tilde{D}(\eta,\xi)=1$ if there exists a unique point $x \in \Z^d$ and a unique direction $i$ such that $\eta(z)=\xi(z)$ for all $z \neq x, x+e_i$ and $\eta(x)-1=\xi(x)$ and $\eta(x+e_i)+1=\xi(x+e_i)$ or $\eta(x)+1=\xi(x)$ and $\eta(x+e_i)-1=\xi(x+e_i)$, and $\tilde{D}(\eta,\xi)=0$ otherwise. A path $\tilde{\Gamma}(\eta,\xi)=(\eta=\eta^0, \eta^1,...,\eta^{m-1},\eta^m=\xi)$ from $\eta$ to $\xi$ is a sequence of configurations $\eta^j$ such that every two consecutive configurations satisfies $\tilde{D}(\eta^j,\eta^{j+1})=1$.  

Consider a family of continuous functions $(\tilde{\mathfrak{u}_x^i})_{1 \le i \le d, x \in \Z^d}$ where $\tilde{\mathfrak{u}_x^i}: \widetilde{\mathcal{H}_{x,x+e_i}} \to \R$. For an ordered pair $(\eta, \xi)$ satisfying $\tilde{D}(\eta,\xi)=1$, define a path integral $I^2_{(\eta,\xi)}$ by  
\begin{displaymath}
 I^2_{(\eta,\xi)}(\tilde{\mathfrak{u}}) := \begin{cases}
	\tilde{\mathfrak{u}_x^i}(\eta)  & \text{if $\eta(x)-1=\xi(x)$ and $\eta(x+e_i)+1=\xi(x+e_i)$} \\	
	-\tilde{\mathfrak{u}_x^i}(\xi)  & \text{if $\eta(x)+1=\xi(x)$ and $\eta(x+e_i)-1=\xi(x+e_i)$}. \\	
\end{cases}
\end{displaymath}
A path integral can be naturally extended to paths of any length as
\begin{displaymath}
I^2_{\tilde{\Gamma}(\eta,\xi)}(\tilde{\mathfrak{u}}):=\sum_{j=0}^{m-1} I^2_{(\eta^j,\eta^{j+1})}(\tilde{\mathfrak{u}}).
\end{displaymath}
A family of continuous functions $(\tilde{\mathfrak{u}_x^i})_{1 \le i \le d, x \in \Z^d}$ is called an $I^2$-closed form if $I^2_{\tilde{\Gamma}(\eta,\xi)}(\tilde{\mathfrak{u}})=0$ hold for all closed paths $\tilde{\Gamma}(\eta,\xi)$.

Now, we construct one-to-one map from the set of $I^1$-closed forms to the set of $I^2$-closed forms. For an $I^1$-closed form $(\mathfrak{u}_x^i)_{1 \le i \le d, x \in \Z^d}$, define a family of continuous functions $(\tilde{\mathfrak{u}_x^i})_{1 \le i \le d, x \in \Z^d}$ as
\begin{align*}
 \tilde{\mathfrak{u}_x^i}(\eta)&:=
	\frac{1}{\sqrt{C_+}}\mathfrak{u}_x^i(\eta) \Psi^{x,x+e_i}_{1,0}(\eta) 		+\frac{1}{\sqrt{C_-}}\mathfrak{u}_x^i(\eta) \Psi^{x,x+e_i}_{0,-1}(\eta) \\
	&+\frac{1}{\sqrt{C_A}}\mathfrak{u}_x^i(\eta) \Psi^{x,x+e_i}_{1,-1}(\eta)
	+\frac{1}{\sqrt{C_C}}\mathfrak{u}_x^i(\eta) \Psi^{x,x+e_i}_{0,0}(\eta),
\end{align*}
then, $(\tilde{\mathfrak{u}_x^i})_{1 \le i \le d, x \in \Z^d}$ is an $I^2$-closed form. On the other hand, for an $I^2$-closed form $(\tilde{\mathfrak{v}_x^i})_{1 \le i \le d, x \in \Z^d}$, define a family of continuous functions $(\mathfrak{v}_x^i)_{1 \le i \le d, x \in \Z^d}$ as
\begin{align*}
\mathfrak{v}_x^i(\eta) &:=
	\sqrt{C_+}\tilde{\mathfrak{v}_x^i}(\eta)  \Psi^{x,x+e_i}_{1,0}(\eta)
	+\sqrt{C_-}\tilde{\mathfrak{v}_x^i}(\eta) \Psi^{x,x+e_i}_{0,-1}(\eta) \\	&+\sqrt{C_A}\tilde{\mathfrak{v}_x^i}(\eta) \Psi^{x,x+e_i}_{1,-1}(\eta) 
	+\sqrt{C_C}\tilde{\mathfrak{v}_x^i}(\eta) \Psi^{x,x+e_i}_{0,0}(\eta) \\
	&- \sqrt{C_E}\Big[(\tilde{\mathfrak{v}_x^i}(\eta^{x=1,x+e_i=-1})+\tilde{\mathfrak{v}_x^i}(\eta^{x=0,x+e_i=0})) \Big] \Psi^{x,x+e_i}_{-1,1}(\eta), 
\end{align*}
then, $(\mathfrak{v}_x^i)_{1 \le i \le d, x \in \Z^d}$ is an $I^1$-closed form.

Let us introduce the notion of an $I^1$-germ of closed form and an $I^2$-germ of closed form. A family of continuous functions $(\mathfrak{g}_i)_{1 \le i \le d}$ (resp. $\tilde{\mathfrak{g}_i}$) where $\mathfrak{g}_i:\mathcal{H}_{0,e_i} \to \R$ is an $I^1$(resp. $I^2$)-germ of closed form if $\mathfrak{u}_x^i:=\tau_x \mathfrak{g}_i$ is an $I^1$(resp. $I^2$)-closed form.

Main theorem of this section is formulated as follows:
\begin{thm}\label{thm:closedform}
For every $I^1$-germ of closed form $\mathfrak{g}$, there exists a sequence of $L^2(\nu_{\rho})$-functions $h_n$ and constants $(c_i)_{1 \le i \le d}$ such that
\begin{displaymath}
\mathfrak{g}_i = \lim_{n \to \infty}(c_i \sum_{j=1}^d (\mathcal{U}^j)_i + \nabla_{0,e_i}\Gamma_{h_n}) \quad \text{in} \quad L^2(\nu_{\rho}) \quad \text{for all} \quad 1 \le i \le d.
\end{displaymath}
\end{thm}

By the one-to-one correspondence between the $I^1$-germ of closed form and the $I^2$-germ of closed form and Cauchy-Schwartz inequality, in order to prove Theorem \ref{thm:closedform} we have only to prove the next theorem:
\begin{thm}
For every $I^2$-germ of closed form $\tilde{\mathfrak{g}}$, there exists a sequence of $L^2(\nu_{\rho})$-functions $h_n$ and constants $(c_i)_{1 \le i \le d}$ such that
\begin{displaymath}
\tilde{\mathfrak{g}}_i = \lim_{n \to \infty}(c_i \sum_{j=1}^d (\tilde{\mathcal{U}^j})_i + \tilde{\nabla}_{0,e_i}\Gamma_{h_n}) \quad \text{in} \quad L^2(\nu_{\rho}) \quad \text{for all} \quad 1 \le i \le d.
\end{displaymath}
where $\tilde{\mathcal{U}}_i(\eta) = 1_{\{\eta(0) \ge 0, \eta(e_i) \le 0\}}$ and\
\begin{align*}
\tilde{\nabla}_{0,e_i}& f(\eta) =\Psi^{0,e_i}_{1,0}(\eta)(f(\eta^{0,e_i})-f(\eta)) + \Psi^{0,e_i}_{0,-1}(\eta)(f(\eta^{0,e_i})-f(\eta)) \\
	&+ \Psi^{0,e_i}_{1,-1}(\eta)(f(\eta^{0=0,e_i=0})-f(\eta)) + \Psi^{0,e_i}_{0,0}(\eta)(f(\eta^{0=-1,e_i=1})-f(\eta)).
\end{align*}
\end{thm}
Applying the method in \cite{KL}, we can deduce this theorem since the proof depends only on the spectral gap estimates, which is proved by Theorem 1.8.1, and the fact that $\nu_{\rho}$ is translation invariant. 

\section{Proof of Theorem \ref{thm:main2}}
The strategy of the proof is the same as given for liquid-solid system in \cite{F}. The main step is to establish the local ergodic theorem.

First, we consider a class of martingales associated with the empirical measure as similar to Case 1. We also use the same notations: $T$, $M^H(t)$, $\pi^N_t$, $Q_{\mu^N}$ and $Q^*$ which are defined in Section 3. In Case 2, we assume the gradient condition, $C_+ + C_- - C_A -2C_E=0$, so the martingale $M^H(t)$ is rewritten as
\begin{displaymath}
M^{H}(t) = \langle \pi^N_t, H \rangle  -  \langle \pi^N_0, H \rangle  
	 - \int^t_0 N^{-d}\sum_{x \in \T^d_N}(\Delta^NH)(\frac{x}{N})\tau_x h(\eta_s)ds
\end{displaymath}
where $h$ is a cylinder function defined by
\begin{align*}
h(\eta) &= C_+ 1_{\{\eta(0)=1\}} - C_- 1_{\{\eta(0)=-1\}}
\end{align*}
and $\Delta^NH$ represents the discrete Laplacian of $H$:
\begin{displaymath}
(\Delta^NH)(\frac{x}{N})=\sum_{i=1}^d N^2[H(\frac{x+e_i}{N})+H(\frac{x-e_i}{N})-2H(\frac{x}{N})].
\end{displaymath}
Following the same argument in Section 3, it is easy to prove that 
\begin{displaymath}
\lim_{N \to \infty}\mathbb{P}_{\mu^N}[\sup_{0 \le t \le T}|M^H(t)| \ge \de]=0
\end{displaymath}
for every $\de>0$, the sequence $\{Q_{\mu^N}, N \ge 1 \}$ is weakly relatively compact and that every limit points $Q^{*}$ is concentrated on absolutely continuous paths $\pi_t(du)=\pi(t,u)du$ with density bounded by 1 and -1 from above and below respectively: $-1 \le \pi(t,u) \le 1$.

From Theorem \ref{thm:uniqueness} stated below, there exists at most one weak solution of (\ref{eq:hydro2}). Therefore, to conclude the proof of the theorem, it remains to show that all limit points of the sequence $\{Q_{\mu^N}, N \ge 1\}$ are concentrated on absolutely continuous trajectories $\pi(t,du)=\pi(t,u)du$ whose densities are weak solutions of the equation (\ref{eq:hydro2}). For this purpose, all we have to show is that
\begin{displaymath}
\lim_{N \to \infty}\mathbb{E}_{\mu^N}[\int^t_0 N^{-d}\sum_{x \in \T^d_N}(\Delta^NH)(\frac{x}{N})\tau_x h(\eta_s)ds]=E_{Q^*}[\int^t_0 \int_{\T^d}(\Delta H)(u)P(\pi(s,u))ds]
\end{displaymath}
where $P$ is the function defined in (\ref{eq:diffusion2}).

First, we establish the local ergodic theorem which enables us to replace the sample mean of microscopic variables with their average under the equilibrium measure having a microscopically defined sample density as its density-parameter. Let $\mu^N_t \in \mathcal{P}(\chi^d_N)$ be the probability distribution of $\eta_t$ on $\chi^d_N$ and let $\tilde{\mu^N}$ be the space-time average of $\{\mu^N_t\}_{0 \le t \le T}$ defined by
\begin{displaymath}
\tilde{\mu^N}=\frac{1}{TN^d}\sum_{x \in \T^d_N} \int^T_0 \mu^N_t \cdot \tau_x^{-1} dt.
\end{displaymath}
Then, we have that
\begin{prop}[local ergodic theorem]
For every cylinder function $f$,
\begin{displaymath}
\lim_{K \to \infty}\limsup_{N \to \infty}E^{\tilde{\mu^N}}[|\bar{f}_{0,K}(\eta)-\langle f \rangle_{\eta^K(0)}|]=0
\end{displaymath}
where $\bar{f}_{0,K}=\frac{1}{(2K+1)^d}\sum_{|x| \le K}\tau_xf(\eta)$.
\end{prop}
\begin{proof}
Let $L$ be an operator on $\mathcal{C}$ defined by $(Lf)(\eta) = \sum_{b \in (\Z^d)^*}L_bf(\eta)$, where $(\Z^d)^*$ stands for the set of all directed bonds.
Following the method of the proof of Theorem 4.1 in \cite{F}, it is easy to show that $\{\tilde{\mu^N}\}_N$ is tight in $\mathcal{P}(\chi^d)$ and an arbitrary limit $\mu \in \mathcal{P}(\chi^d)$ satisfies $\mu(Lf)=0$ for every cylinder function $f$. Moreover, by definition, $\mu$ is invariant under spatial translations.
Therefore, we have that support $(\mu) \subset \{-1,0\}^{\Z^d} \cup \{0,1\}^{\Z^d}$. It is known that the translation-invariant $L$-stationary measure on $\{-1,0\}^{\Z^d}$ or $\{0,1\}^{\Z^d}$ is a superposition of Bernoulli product measures. Then, the law of large numbers concludes the proposition.
\end{proof}

Next, we need to prove that the sample density defined microscopically can be replaced in the limit with the macroscopic one. We can use Young measures to complete it by the exactly same way as in \cite{F}. Then, combining with these results, the main theorem will be concluded. The details are omitted.

\section{Proof of Theorem \ref{thm:main3}}
In Case 3, since the hydrodynamic equation is the heat equation, the replacements which are required in Case 2 are unnecessary. So, this is the easiest case to prove the hydrodynamic limit.

First, we consider a class of martingales associated with empirical measure again. Then, as in Case 2, the martingale $M^H(t)$ is rewritten as
\begin{align*}
M^{H}(t) &  = \langle \pi^N_t, H \rangle  -  \langle \pi^N_0, H \rangle - C_E \int^t_0  \langle \pi^N_s, \Delta^N H \rangle ds \\
	& + \frac{C_+ - C_-}{2} \int^t_0 N^{-d}\sum_{x \in \T^d_N} 1_{\{\eta_s(x)=0\}}(\Delta^NH)(\frac{x}{N})ds
\end{align*}
where $\Delta^NH$ represents the discrete Laplacian of $H$.

Following the same argument in Section 3 and 4, all we have to show is that
\begin{equation}
\lim_{N \to \infty}\mathbb{E}_{\mu^N}[\int^t_0 N^{-d}\sum_{x \in \T^d_N} 1_{\{\eta_s(x)=0\}}ds]=0. \label{eq:localergodic3}
\end{equation}

With the notation defined in the last section, (\ref{eq:localergodic3}) is rewritten as
\begin{displaymath}
\lim_{N \to \infty}E^{\tilde{\mu^N}}[1_{\{\eta(0)=0\}}]=0.
\end{displaymath}

Following the method of the proof of Theorem 4.1 in \cite{F} again, it is easy to show that $\{\tilde{\mu^N}\}_N$ is tight in $\mathcal{P}(\chi^d)$ and an arbitrary limit $\mu \in \mathcal{P}(\chi^d)$ satisfies $\mu(Lf)=0$ for every cylinder function $f$. Moreover, by definition, $\mu$ is invariant under spatial translations. Therefore, we have that support $(\mu) \subset \{-1,1\}^{\Z^d}$ and it concludes the theorem.

\section{Weak Solutions of Nonlinear Parabolic Equations}
In this section, we fix the terminology of weak solutions of parabolic equations and present the uniqueness of such equations. Hereafter $\phi:\R \to \R$ is a strictly increasing and Lipschitz continuous function. We consider the Cauchy problem:
\begin{equation}
\left\{
\begin{aligned}
\partial_{t}\rho(t,u) & = \De(\phi(\rho(t,u))) \ \Big(= \sum_{i=1}^d \frac{\partial^2}{\partial u_i^2} \phi(\rho(t,u)) \Big) \\
\rho(0,\cdot) & = \rho_0(\cdot), \label{eq:cauchy}
\end{aligned}
\right.
\end{equation}
and define weak solutions of this Cauchy problem.

\begin{defn}
Fix a bounded initial profile $\rho_0:\T^d \to \R$. A measurable function $\rho \equiv \rho(t)=\rho(t,u) \in C([0,T], \M(\T^d)) \cap L^2([0,T] \times \T^d)$ is a weak solution of the Cauchy problem (\ref{eq:cauchy}) if for every function $H:\T^d \to \R$ of class $C^2(\T^d)$ and for every $0 \le t \le T$
\begin{displaymath}
\int_{\T^d} \ H(u)\rho(t,u)du = \int^t_0 ds \int_{\T^d} \phi(\rho(s,u))\De H(u)du \\
+ \int_{\T^d} \ H(u)\rho_0(u) du.
\end{displaymath}
\end{defn}
We prove the uniqueness of weak solutions in this class.
\begin{thm}\label{thm:uniqueness}
Fix a bounded measurable function $\rho_0:\T^d \to \R$. There exists at most one weak solution of the parabolic equation (\ref{eq:cauchy}).
\end{thm}
\begin{proof}
We can apply the proof of Theorem A.2.4.4 in \cite{KL} straightforwardly.
\end{proof}

\subsection*{Acknowledgement}
The author would like to thank Professor T. Funaki for helping her with valuable suggestions.


\begin{thebibliography}{99}
\bibitem{CDFG}{\sc P.\ Collet, F.\ Dunlop, D.\ Foster, T.\ Gobron},
{\it Product measures and front dynamics for solid on solid interfaces}, J.\ Stat.\ Phys., {\bf 89} (1997), 509--536.

\bibitem{F}{\sc T.\ Funaki},
{\it Free boundary problem from stochastic lattice gas model}, Ann.\ Inst.\ H.\ Poincar\'e,\ Probab.\ Statis., {\bf 35} (1999), 573--603.

\bibitem{KL}{\sc C.\ Kipnis and C.\ Landim},
{\it Scaling Limits of Interacting Particle Systems}, 
1999, Springer.

\bibitem{KLO}{\sc C.\ Kipnis, C.\ Landim and S.\ Olla},
{\it Hydrodynamic limit for a nongradient system: The generalized symmetric exclusion process}, 
Comm.\ Pure Appl.\ Math., {\bf 47} (1994), 1475--1545.

\bibitem{KV}{\sc C.\ Kipnis, and S.R.S.\ Varadhan},
{\it Central limit theorem for additive functionals of reversible Markov processes and applications to simple exclusion}, 
Comm.\ Math.\ Phys., {\bf 104} (1986), 1--19.

\bibitem{Q}{\sc J.\ Quastel},
{\it Diffusion of color in the simple exclusion process}, 
Comm.\ Pure Appl.\ Math., {\bf 45} (1992), 623--679.

\bibitem{S}{\sc H.\ Spohn},
{\it Large Scale Dynamics of Interacting Particles}, 
1991, Springer.

\end{thebibliography}
\end{document}